\documentclass[12pt,reqno,a4paper]{amsart}
\usepackage[centertags]{amsmath}
\usepackage{amsfonts}
\usepackage{amssymb}
\usepackage{amsthm}
\usepackage{amsmath}
\usepackage{times}
\usepackage[top=2.5cm, bottom=2.5cm, left=2.5cm, right=2.5cm]{geometry}
\usepackage{color}
\usepackage{graphicx}
\usepackage{hyperref}
\hypersetup{
    colorlinks=true,
    linkcolor=blue,
    filecolor=magenta,
    urlcolor=cyan,
}

\numberwithin{equation}{section}

\newcommand{\cC}{\mathcal C}
\newcommand{\cB}{\mathcal B}
\newcommand{\cA}{\mathcal A}

\newcommand{\cH}{\mathcal{H}}

\newcommand{\cL}{\mathcal{L}}

\newcommand{\cG}{\mathcal{G}}
\newcommand{\cT}{\mathcal{T}}

\newcommand{\fM}{\mathfrak{M}}

\newcommand{\K}{\mathrm{Op}}
\newcommand{\cM}{{\mathcal M}}

\newcommand{\N}{\mathbb{N}}

\newcommand{\Z}{\mathbb{Z}}

\newcommand{\R}{\mathbb{R}}
\newcommand{\Q}{\mathbb{Q}}
\newcommand{\HD}{\mathbb{HD}}

\renewcommand{\P}{\mathbb{P}}
 
\newcommand{\hull}[1]{\mathrm{co}(#1)}

\newcommand{\atw}{\mathfrak{A}}
\newcommand{\funcforce}{\mathfrak{F}_H}
\newcommand{\func}{\mathfrak{F}}

\newcommand{\Borel}{B}

\renewcommand{\epsilon}{\varepsilon}

\newcommand{\sdist}{\tilde{\mathrm d}}%\mathop{\mathrm{sdist}}}
\newcommand{\dist}{\mathrm{d}}% \mathop{\mathrm{dist}}}
\newcommand{\distance}{\mathrm{d}}

\newcommand{\loc}{\mathrm{loc}}

\newcommand{\map}{\Psi}

\newcommand{\strictlyincluded}{\subset\subset}

\newcommand{\res}{\mathop{\hbox{\vrule height 7pt width 0.5pt depth 0pt
\vrule height 0.5pt width 6pt depth 0pt}}\nolimits}

\newcommand{\p}{\partial}

\newcommand{\Per}{\mathop{\mathrm{Per}}}

\newcommand{\diam}{\mathop{\mathrm{diam}}}
\renewcommand{\div}{\mathop{\mathrm{div}}}

\theoremstyle{plain}
\newtheorem{theorem}{Theorem}[section]
\newtheorem{lemma}[theorem]{Lemma}
\newtheorem{definition}[theorem]{Definition}
\newtheorem{proposition}[theorem]{Proposition}
\newtheorem{corollary}[theorem]{Corollary}

\theoremstyle{definition}

\newtheorem{remark}[theorem]{Remark}

\date{\today}

\begin{document}

\title[GMM for partitions]{Minimizing movements for mean curvature flow
of partitions}

\author[G. Bellettini, Sh.Yu. Kholmatov]{Giovanni
Bellettini$\!^{1,2}$, Shokhrukh Kholmatov$\!^{2,3,4}$}

\address{$^1\!$Universit\'a degli studi di Siena,
Dipartimento di Ingegneria
dell'Informazione e Scienze Matematiche, 
via Roma 56,
53100 Siena, Italy}

\email{$^{1,2}\!$bellettini@diism.unisi.it}

\address{$^2\!$International Centre for Theoretical Physics (ICTP),
Strada Costiera 11, 34151 Trieste, Italy}
\address{$\!^3\!$Scuola Internazionale Superiore di
Studi Avanzati (SISSA)\\
Via Bonomea 265, 34136 Trieste, Italy}
\address{$\!^4\!$Universit\"at Wien 
Fakult\"at f\"ur Mathematik\\
Oskar-Morgenstern-Platz, 1, 1090, Vienna, Austria}

\email{$^{2,3,4}\!$shokhrukh.kholmatov@univie.ac.at}
%\phone{+393275843010}

\begin{abstract}
We prove the existence of a weak global in time mean curvature flow 
of a bounded partition  of 
space using the method of minimizing movements. 
The result is extended to the case 
when suitable driving forces are present. 
We also prove some consistency results for a
minimizing movement solution with smooth and 
viscosity solutions when the evolution 
starts from a partition made by a union of bounded 
sets at a positive distance. In addition, the
motion starting from the union of convex sets at  a positive distance 
agrees with the classical mean curvature flow and is stable with respect to 
the Hausdorff convergence of the initial partitions.
\end{abstract}


\maketitle

%\tableofcontents

\jot=12pt

%%%%%%%%%%%%%%%%%%%%%%%%%%%%%%
\section{Introduction}
%%%%%%%%%%%%%%%%%%%%%%%%%%%%%%

Mean curvature evolution  of partitions became popular in
recent years because of its applications in material science  and
physics, especially evolutions of grain boundaries and 
motion of immiscible fluid systems, 
see e.g. \cite{BKPS:1999,Br:1978,KL:2001,MNPS:2016} 
and references therein.
Behaviour of the motion in the two phase case, i.e. 
in the case of classical motion by mean curvature of a boundary 
as a gradient flow of the area functional,  
is rather well-understood,
see for instance \cite{Bell:2012,CM:book,Eck:2004,GH:86,Giga:06,Hui:84,Man:2011} 
and references therein. 

Mean curvature evolution of interfaces in the multiphase case in general involves
motion of surface junctions in $\R^n,$ or triple and multiple points in the plane,
an already nontrivial problem. 
We refer to the  survey \cite{MNPS:2016} and references therein for
recent results on curvature evolution  of planar networks. 

Not much seems to be known in higher space dimensions; 
short time existence of the motion of subgraph-type partitions  
has been derived in \cite{Fr:2010,Fr:20-2} and
 well-posedness and short time existence
of the motion by mean curvature of three surface clusters have 
been recently shown in \cite{DGK:14}.

Even in the two phase case, 
the classical flow describes the  motion only up to the 
appearance of the first singularity. In order to continue the 
motion through singularities, several notions of generalized solutions
have been suggested: Brakke varifold-solution \cite{Br:1978}, 
the viscosity solution (see \cite{Giga:06} and references therein),
the Almgren-Taylor-Wang \cite{ATW93} 
and Luckhaus-Sturzenhecker 
\cite{LS:95}  solutions, the minimal  barrier solution
(see \cite{Bell:2012} and references therein); we also
refer to \cite{ESS:1992, Il:94} for other types of solutions. 
At the moment the lack of the comparison principle  in the multiphase case  
results in a lot of difficulties to extend 
such notions as viscosity and barrier solutions,  while besides Brakke solution,
some
other generalized solutions have been successfully extended
to partitions. For example, the authors 
of \cite{LO:2016} have proved
the existence of a distributional solution  of mean curvature evolution of 
partitions on the torus using the time thresholding method introduced in
\cite{MBO:1992}, see also \cite{EO:2014,MBO:1994}; furthermore the authors 
of \cite{KT:2016} showed  the existence of a Brakke
flow. 

In \cite{DG:93} De Giorgi generalized the Almgren-Taylor-Wang and 
Luckhaus-Sturzen\-hecker approach
to what he called the minimizing movements method. 
In the present paper, we prove the existence of a generalized minimizing 
movement solution in $\P_b(N+1),$ the collection of all 
partitions of $\R^n,$ $n\ge2,$ having $N+1 \ge2$ components, 
with the first $N$-components
bounded. This is the multiphase generalization
of the evolution of a compact boundary in the two-phase case
($N=1$), for which 
the generalized  minimizing movement solution has 
been introduced and studied in \cite{ATW93,LS:95}.

Let us recall the definition (see \cite{DG:93,DG-96}, 
also \cite{LAln,AGS:05}).

\begin{definition}[\bf Generalized minimizing 
movement for partitions]\label{def:GMM}
Let $\P_b(N+1)$ be the set of all bounded $(N+1)$-partitions 
of $\R^n$ (Definition \ref{def:g_partitions}) 
endowed with the $L^1(\R^n)$-convergence, and let
$\func: \P_b(N+1)\times \P_b(N+1)\times [1,+\infty) 
\to[-\infty,+\infty]$ be defined as 
$$
\func(\cA,\cB;\lambda) = \Per(\cA) +\frac{\lambda}{2}
\sum\limits_{j=1}^{N+1} \int_{A_j\Delta B_j} 
\dist(x,\p B_j)dx,\quad\cA,\cB\in \P_{b}(N+1),
$$
where $\Per(\cA)=\frac12\sum\limits_{j=1}^{N+1} P(A_j)$ is the  
perimeter of the partition  $\cA=(A_1,\ldots,A_{N+1})$ and
$\dist(\cdot,E)$ is the distance function from 
$E\subseteq\R^n.$   We say that a map $\cM:[0,+\infty)\to \P_b(N+1)$
is a generalized minimizing movement (shortly a $GMM$) 
associated to $\func$  starting from
$\cG\in \P_b(N+1)$ and we write $\cM\in GMM(\func,\cG),$ if there exist
$\cL:[1,+\infty)\times \N_0 \to \P_b(N+1)$ and a 
diverging sequence $\{\lambda_h\}$ such that
$$
\lim\limits_{h\to+\infty} 
\cL(\lambda_h,[\lambda_ht]) = \cM(t)\quad \text{in $L^1(\R^n)$
for any $t\ge0,$}
$$
where the bounded partitions $\cL(\lambda,k),$ $\lambda\ge1,$ $k\in\N_0,$ 
are defined  inductively as $\cL(\lambda,0)=\cG$  and
$$
\func(\cL(\lambda,k+1),\cL(\lambda,k);\lambda) = 
\min\limits_{\cA\in \P_b(N+1)} \func(\cA,\cL(\lambda,k);\lambda)
\qquad \forall k \geq 0.
$$
When $GMM(\func,\cG)$ is a singleton, it is called the minimizing movement
starting from $\cG$ and denoted by $MM(\func,\cG).$  
\end{definition}

We shall also consider GMM associated to the functional
$$
\funcforce(\cA,\cB;\lambda) = 
\Per(\cA) + \sum\limits_{j=1}^{N+1} \int_{A_j} H_jdx
+\frac{\lambda}{2}\sum\limits_{j=1}^{N+1} \int_{A_j\Delta B_j} 
\dist(x,\p B_j)dx,\quad\cA,\cB\in \P_{b}(N+1) 
$$
for suitable driving forces $H_i,$ $i=1,\ldots,N+1$ 
(see Section \ref{sec:prescribed_curvature}).

Our main result is the following (see Theorems 
\ref{teo:existence_GMM} and  \ref{teo:existence_GMM_2}
for the precise 
statements): %

\begin{theorem}\label{teo:main_introduction}
For any $\cG\in \P_b(N+1),$ $GMM(\func,\cG)$ is nonempty, i.e. 
there exists a generalized minimizing movement starting 
from $\cG.$ Moreover, 
\begin{itemize}
 \item[1)] any such movement $\cM(t) = (M_1(t),\ldots, M_{N+1}(t))$ is locally 
$\frac{1}{n+1}$-H\"older continuous in time;
\item[2)] $\bigcup\limits_{j=1}^N M_j(t)$ is contained in the closed convex envelope of 
the union $\bigcup\limits_{j=1}^N G_j$ of the bounded components of $\cG$ for any $t>0.$
\end{itemize}
Finally, similar results are valid for $\funcforce.$ 
\end{theorem}

To prove Theorem \ref{teo:main_introduction}
we establish  uniform density estimates for  minimizers of 
$\func $ and $\funcforce.$ A lower-type density estimate 
for minimizers of $\func $
could be proven using the slicing method for currents as 
in the thesis \cite{Car:Th}, or also using the infiltration technique of 
\cite[Lemma 4.6]{LT:2002} (see also \cite[Section 30.2]{Mag12}).
In Section \ref{sec:partitions} we prove that $(\Lambda,r_0)$-minimizers
of $\Per$ in $\R^n$ (Definition \ref{def:almost_min}) satisfy uniform 
density estimates using 
the method of cutting out and filling in with balls, 
an argument of \cite{LS:95}.  
%We recall that the upper density estimate  for the 
%$(\Lambda,r_0)$-minimizer $\cA$ can be made independent of $N.$
%Moreover, according to \cite[Lemma 4.6]{LT:2002} there exists 
%$r_1>0$ and $\eta>0$ such that for every ball $B_r(x),$ $r\in (0,r_1),$ 
%such that $|B_r(x)\setminus \bigcup\limits_{j=3}^N A_j|<\eta r^n$ one has 
%$|B_{r/2}(x)\setminus \bigcup\limits_{j=3}^N A_j|=0.$
%Hence,  $B_{r/2}(x)$ is essentially covered by $A_1$ and $A_2,$
%therefore the density estimates in two-phase ($N=2$\!) is applicable,
%so that the lower density estimate can be made 
%independent of $N.$ However, the maximum range $r_1$
%of radii satisfying such lower density estimate depends 
%on the minimizer wich could be very small with respect to
%$\hat r_0:=\max\{r_0,\frac{n}{4(N-1)\Lambda}\}$ of 
%Theorem \ref{teo:density_est}. Moreover, in the application  
%(see the proofs of Theorem \ref{teo:existence_GMM} and 
%\ref{teo:existence_GMM_2}), 
%we use $\hat r_0:=\frac{C}{\lambda},$
%where $C:=C(n,N,\cG)$ depends only on $n,$ $N$ and the diameter of
%the union of bounded components of initial partition
%$\cG.$

Some consistency results of GMM starting from disjoint partitions 
(Definition \ref{def:disjoint}) with other notions of solutions 
are shown in Section 6. In particular we have:

\begin{theorem}\label{teo:consistency_intro}
a) Let $\cG\in\P_b(N+1)$ be a disjoint partition and
suppose that for each $i=1,\ldots,N$ there exists a family of smooth sets 
$L_i(t),$ $t\in [0,t_o),$
whose boundaries evolve smoothly by mean curvature in $[0,t_o)$ such that 
$L_i(0)=G_i.$ Then for any $\cM\in GMM(\func,\cG)$ we have
$$
M_i(t) = L_i(t),\qquad t\in [0,t_o).
$$
\smallskip

b) Let $\cG\in\P_b(N+1)$ be a disjoint partition such that
for each $i=1,\ldots,N,$  $|\p G_i|=0,$ 
and suppose that the viscosity solution $v_i$ \cite{Ch:2005} of 
$$
\frac{\p u}{\p t} = |\nabla u|\,\div\,\frac{\nabla u}{|\nabla u|}
$$
starting from $\chi_{G_i^c} - \chi_{G_i}$ is unique. Then 
$GMM(\func,\cG)=\{\cM\}$ is a singleton and 
$$
v_i(x,t) = \chi_{M_i(t)^c}(x) - \chi_{M_i(t)}(x)
\qquad \text{ for every $(x,t)\in \R^n\times[0,+\infty).$}
$$ 
\end{theorem}

In Theorem  \ref{teo:disjoint_initial_part} 
we also show the following  stability result.

\begin{theorem}\label{teo:convex_introduction}
Suppose that 
$\cC = (C_1,\dots,C_{N+1})\in \P_b(N+1),$
where $C_1, \ldots,C_N$ are convex sets whose closures are disjoint. 
Then the GMM %generalized minimizing movement 
associated to $\func $ and starting from $\cC$ 
is the minimizing movement $\{\cM\}=MM(\func,\cC),$ and writing 
$$
\cM(t) = (M_1(t),\ldots,M_{N+1}(t)),
$$
we have that each $M_i(t)$ 
agrees  with the classical mean curvature flow starting from $C_i$,
up to the  extinction time. Moreover, if a sequence $\{\cG^{(k)}\}\subset 
\P_b(N+1)$
converges to $\cC\in \P_b(N+1)$ in the Hausdorff distance,
then any $\cM^{(k)}\in GMM(\func,\cG^{(k)})$ 
converges to  $\{\cM\}=MM(\func,\cC)$ in the Hausdorff distance at  
every time $t\ge0.$ 
\end{theorem}

The proof of the consistency with the classical mean curvature flow
relies on the results of 
\cite{BCCN:05}, while for the stability in the Hausdorff distance
we employ the comparison results from \cite{BH:2016,CMP:2015}.

Our results do not apply to the case when at least two components of
a partition are unbounded, since in this case they have infinite
perimeter, and it also may happen that the right hand side of 
\eqref{eq:pro_distance},  which allows to replace 
$\int_{E_i \Delta F_i}\dist(x,\p F_i)dx$ with the signed
distance function, is not well-defined.

The plan of the paper is the following.

In Section \ref{sec:notation} we set the notation and recall  some results from
the theory of finite perimeter sets. Section \ref{sec:partitions} 
is devoted to the definitions of partitions and 
density estimates for $(\Lambda,r_0)$-minimizers.
In Section \ref{sec:existence_GMM} we prove the existence of minimizers
of $\func $  in $\P_b(N+1)$ (Theorem \ref{teo:existence_of_minimizers}),
the density estimates (Theorem \ref{teo:density_est_ATW}), 
and -- one of our main results -- the existence of $GMM$ for $\func $ 
(Theorem \ref{teo:existence_GMM}).
The existence of $GMM$ for $\funcforce$ is shown in 
Section \ref{sec:prescribed_curvature}.
%Some comparison results will be recalled in Section \ref{sec:comparison_2_phase}
%and 
Finally, in Section \ref{sec:evolution_disjoint_part} 
we show that any GMM starting from a disjoint partition is also disjoint 
and prove Theorem \ref{teo:consistency_intro} -- 
the consistency result with smooth mean curvature flow.
As a nontrivial application 
of these facts, we show the consistency and stability 
results stated in Theorem \ref{teo:convex_introduction}.

%%%%%%%%%%%%%%%%%%%%%%%%%%%%%%%%%%%%%%%%%%%%%%%%%%%%%%%%%%%%%%
\section{Notation and preliminaries}\label{sec:notation}
%%%%%%%%%%%%%%%%%%%%%%%%%%%%%%%%%%%%%%%%%%%%%%%%%%%%%%%%%%%%%%

In this section we introduce the notation 
 and collect 
some important properties of sets of 
locally finite perimeter. The standard references 
for $BV$-functions and sets of finite perimeter are 
\cite{AFP:00,Gius84}.

We use $\N_0$ to denote the set of all nonnegative integers. 
Given a finite subset $I\subset \N_0,$ we write $|I|$ for 
the number of elements of $I.$ 
The symbol $B_r(x)$  stands for the open ball in 
$\R^n$  centered at $x\in\R^n$ 
of radius $r>0.$ The characteristic function 
of a Lebesgue measurable set $F$
is denoted by $\chi_F$ and its Lebesgue measure
by $|F|;$ we set also $\omega_n:=|B_1(0)|.$
We denote by $E^c$ the complement of $E$ in $\R^n.$

$\K(\R^n)$ (resp. $\K_b(\R^n)$) is  the collection of all 
open (resp. open and bounded) subsets of $\R^n.$ 
The set of $L_\loc^1(\R^n)$-functions
having locally bounded total variation in $\R^n$ is denoted by 
$BV_\loc(\R^n)$ and the elements of 
$$
BV_\loc(\R^n,\{0,1\}):=\{E\subseteq\R^n:\,\, \chi_E\in BV_\loc(\R^n)\}
$$
are called locally finite perimeter sets. 
Given a $E\in BV_\loc(\R^n,\{0,1\})$ we denote by
\begin{itemize}
\item[a)]  $P(E,\Omega):=\int_\Omega|D\chi_E|$ the perimeter of $E$
in $\Omega\in \K(\R^n);$
\item[b)] $\p E$ the measure-theoretic boundary of $E:$
$$
\p E:=\{x\in \R^n:\,\, 0<|B_\rho\cap E|<|B_\rho|\quad\forall \rho>0\};
$$
\item[c)] $\p^*E$ the reduced boundary of $E;$
\item[d)] $\nu_E$   the outer generalized unit normal to $\p^*E.$
\end{itemize}
For simplicity, we set $P(E):=P(E,\R^n)$  provided $E\in BV(\R^n,\{0,1\}).$
Further, given a Lebesgue measurable set $E\subseteq\R^n$ and 
$\alpha\in [0,1]$ we define 
$$
E^{(\alpha)}:=\left\{x\in\R^n:\,\,\lim\limits_{\rho\to0^+}
\frac{|B_\rho(x)\cap E|}{|B_\rho(x)|} =\alpha\right\}.
$$
Unless otherwise stated, we always suppose that any locally finite perimeter 
set $E$ we consider coincides 
with $E^{(1)}$ (so that by \cite[Proposition 3.1]{Gius84} 
$\p E$ coincides with the topological boundary). We recall that $\overline{\p^*E}=\p E$  and 
$D\chi_E = \nu_E d\cH^{n-1}\res \p^*E,$  where $\cH^{n-1}$ is the $(n-1)$-dimensional
Hausdorff measure in $\R^n$ and $\res$ is the symbol of restriction.
Given a nonempty set $E\subseteq\R^n,$  
$\dist(\cdot,E)$ stands for the distance function from $E$
and 
$$
\sdist(x,\p E) = \dist(x,E) - \dist(x,\R^n\setminus E)
$$
is the signed distance function from $\p E,$ negative inside $E.$  
We also write $\dist(A,B)$ to denote the distance between 
$A, B\subset\R^n.$

\begin{theorem}\cite{DG:61}\label{cor:density}
Let $E\in BV_\loc(\R^n,\{0,1\}).$ Then for any $x\in \p^*E$
$$
\lim\limits_{\rho\to0^+} \frac{|E\cap B_\rho(x)|}{|B_\rho(x)|}
=\frac12,\qquad 
\lim\limits_{\rho\to0^+} \frac{P(E,B_\rho(x))}{\omega_{n-1}r^{n-1}} =1.
$$
\end{theorem}

\begin{theorem}\cite[Theorem 3.61]{AFP:00} \label{cor:support_of_Dchi}
For every  $E\in BV_\loc(\R^n,\{0,1\})$ 
$$
\cH^{n-1}(\R^n\setminus (E^{(0)} \cup E  \cup \p^*E)) = 0.
$$
Moreover, $\cH^{n-1}(E^{(1/2)}\setminus \p^*E)=0.$
\end{theorem}

\begin{remark}
Given $E\in BV_\loc(\R^n,\{0,1\})$ the map 
$\Omega\in \K(\R^n)\mapsto P(E,\Omega)$ extends to a Borel measure in $\R^n,$
so that $P(E,\Borel) = \cH^{n-1}(\Borel\cap \p^*E)$ for every Borel set $\Borel\subseteq\R^n.$
\end{remark}

\begin{theorem}\cite[Theorem 16.3]{Mag12}\label{teo:Gauss-green_meas}
If $E$ and $F$ are sets of locally finite perimeter, and we let
$$
\{\nu_E = \nu_F \} = \{x \in \p^*E \cap \p^*F:\,\, \nu_E(x) = \nu_F(x)\},
$$
$$
\{\nu_E = -\nu_F \} = \{x \in \p^*E \cap \p^*F:\,\, \nu_E(x) = - \nu_F(x)\},
$$
then $E \cap F,$ $E \setminus F$ and $E \cup F$ are locally finite perimeter sets
with
\begin{equation}\label{ess_intersection}
\p^*(E \cap F) \approx \big(F  \cap \p^* E\big) \cup 
\big(E  \cap \p^* F\big) \cup \big\{\nu_E = \nu_F\big\}, 
\end{equation}
\begin{equation}\label{ess_differense}
\p^*(E \setminus F) \approx \big(F^{(0)} \cap \p^* E\big) \cup 
\big(E  \cap \p^* F\big) \cup \big\{\nu_E = -\nu_F\big\}, 
\end{equation}
\begin{equation}\label{ess_union}
\p^*(E \cup F) \approx \big(F^{(0)} \cap \p^* E\big) \cup 
\big(E^{(0)} \cap \p^* F\big) \cup \big\{\nu_E = \nu_F\big\},
\end{equation}
where $A\approx B$  means   $\cH^{n-1}(A\Delta B) =0.$
Moreover, for every Borel set $\Borel \subseteq R^n$
\begin{equation}\label{per_intersection}
P(E \cap F, \Borel) = P(E, F  \cap \Borel) + P(F, E  \cap \Borel)+ 
H^{n-1}\big(\{\nu_E = \nu_F\} \cap \Borel\big), 
\end{equation}
\begin{equation}\label{per_difference}
P(E \setminus F, \Borel) = P(E, F^{(0)} \cap \Borel) + P(F, E  \cap \Borel)+ 
H^{n-1}\big(\{\nu_E = -\nu_F\} \cap \Borel\big), 
\end{equation}
\begin{equation}\label{per_union}
P(E \cup F, \Borel) = P(E, F^{(0)} \cap \Borel) + P(F, E^{(0)} \cap \Borel)+ 
H^{n-1}\big(\{\nu_E = \nu_F\} \cap \Borel\big). 
\end{equation}
\end{theorem}

Finally, recall that for every $E,F\in BV_\loc(\R^n,\{0,1\})$ and $\Omega\in \K(\R^n)$
\begin{equation}\label{famfor}
P(E\cap F,\Omega) + P(E\cup F,\Omega) \le P(E,\Omega) + P(F,\Omega).
\end{equation}
 
%%%%%%%%%%%%%%%%%%%%%%%%%%%%%%%%%%%%%%%%%%%%%%
\section{Partitions}\label{sec:partitions}
%%%%%%%%%%%%%%%%%%%%%%%%%%%%%%%%%%%%%%%%%%%%%%

Now we give the notions of partition,
$(\Lambda,r_0)$-minimizer  and boun\-ded partition.
The main result  of this section is represented by the 
density estimates for $(\Lambda,r_0)$-minimizers
(Theorem \ref{teo:density_est}).

\begin{definition}[\bf Partition]\label{def:partitions}
Given an integer  $N\ge2,$  an $N$-tuple  $\cC=(C_1,\ldots,\allowbreak C_{N})$
of subsets of  $\R^n$ is called an {\bf $N$-partition} 
  of $\R^n$ (a partition, for short) if 
\begin{itemize}
\item[(a)] $C_i\in BV_\loc(\R^n,\{0,1\})$ for every $i=1,\ldots, N,$
\item[(b)] $\sum\limits_{i=1}^{N} |C_i\cap K| =|K|$  for each 
compact $K\subset \R^n.$
\end{itemize}
\end{definition}

The collection of all $N$-partitions of $\R^n$ 
is denoted by $\P(N).$
Our assumptions $C_i=C_i^{(1)}$ implies $C_i\cap C_j=\emptyset$
for $i\ne j.$ 
Notice also that we do not exclude the case $C_i=\emptyset.$

The elements of $\P(N)$ are denoted by calligraphic 
letters  $\cA,\cB,\cC,\ldots$ and the
components of $\cA\in \P(N)$ by the
corresponding roman letters 
$(A_1,\ldots,A_N).$ 
 The functional 
$$
(\cA,\Omega)\in \P(N)\times \K(\R^n)\mapsto \Per(\cA,\Omega):= 
  \frac12\sum\limits_{j=1}^{N} P(A_j,\Omega)
$$
is called the perimeter  of the partition $\cA$ in $\Omega.$ For simplicity,
we write $\Per(\cA):=\Per(\cA,\R^n).$ 
We set
$$
\cA\Delta\cB:= \bigcup\limits_{j=1}^{N} A_j\Delta B_j 
$$
and
\begin{equation}\label{L_1dist}
|\cA\Delta \cB|:= \sum\limits_{j=1}^{N} |A_j\Delta B_j|, 
\end{equation}
where $\Delta $ is the symmetric difference of sets, i.e. 
$E\Delta F = (E\setminus F)\cup (F\setminus E).$

We say that the sequence $\{\cA^{(k)}\}\subseteq\P(N)$ 
{\it converges} to $\cA\in \P(N)$ in $L_\loc^1(\R^n)$  if 
$$
|(\cA^{(k)}\Delta \cA)\cap K|:= \sum\limits_{j=1}^{N} |(A_j^{(k)}\Delta A_j)\cap K| \to 0
\qquad\text{as $k\to+\infty$}
$$
for every compact set $K\subseteq \R^n.$
Since $E\in BV_\loc(\R^n,\{0,1\})\mapsto P(E,\Omega)$
is $L_\loc^1(\R^n)$-lower semicontinuous 
for any $\Omega\in \K(\R^n),$ the map $\cA\in \P(N)\mapsto \Per(\cA,\Omega)$ 
is $L_\loc^1(\R^n)$-lower semicontinuous.
The following compactness result  can be proven 
using  \cite[Theorem 3.39]{AFP:00} and a diagonal argument.
 
\begin{theorem}[\bf Compactness]\label{teo:compactness}
Let $\{\cA^{(l)}\}\subset \P(N)$  be a sequence of partitions such that 
\begin{equation}\label{uniform_p}
\sup\limits_{l\ge1} \Per(\cA^{(l)},\Omega)<+\infty\qquad\forall \Omega\in\K_b(\R^n). 
\end{equation}
Then there exist a partition $\cA\in \P(N)$ and a subsequence
$\{\cA^{(l_k)}\}$ converging
to $\cA$ in $L_\loc^1(\R^n)$ as $k\to+\infty.$ 
\end{theorem}

The next result is proven for the convenience of the reader.

\begin{proposition}[\bf Boundaries of ``neighboring'' sets]\label{teo:boundary_neighbor}
Let $\cA\in \P(N).$ Then 
$$
\cH^{n-1}\Big(\p^*A_i\setminus  
\bigcup\limits_{j=1,\,j\ne i}^N \p^*A_j \Big) =0\qquad \forall i=1,\ldots,N.
$$
\end{proposition}

\begin{proof}
If $N=2,$ then 
\begin{equation}\label{nteng2}
\p^*A_1 = \p^*(\R^n\setminus A_1)  = \p^*A_2, 
\end{equation}
hence we suppose $N\ge3.$ 
There is no loss of generality in assuming $i=1.$ 
By virtue of \eqref{ess_union}, 
there exists an $\cH^{n-1}$-negligible set 
$Z_{2;3}^{}\subset\p A_2\cup\p A_3$ such that 
$$
\p^*(A_2\cup A_3) = Z_{2;3}^{}\cup \Big(A_2^{(0)}\cap \p^*A_3\Big)
\cup \Big(A_3^{(0)}\cap \p^*A_2\Big)\cup \{\nu_{A_2} = \nu_{A_3}\}.
$$
Therefore, 
$$
\p^*(A_2\cup A_3)\subseteq Z_{2;3}^{}\cup \Big(\p^*A_2 \cup \p^*A_3\Big), 
$$
and by an induction argument, for any $j\in \{3,\ldots,N\}$
there exists an $\cH^{n-1}$-negligible set 
$Z_{2,\ldots,j-1;j}^{}\subset \p 
\Big(\bigcup\limits_{h=2}^{j-1}A_h\Big)\cup 
\p A_j$ such that 
$$
\p^*\Big(\bigcup\limits_{j=2}^N A_j\Big)
\subseteq \Big(\bigcup\limits_{j=3}^{N} Z_{2,\ldots,j-1;j}^{}\Big)\cup
\Big(\bigcup\limits_{j=2}^N \p^*A_j \Big).
$$
Hence
\begin{equation}\label{birlashmalar}
\p^*A_1 \setminus  \bigcup\limits_{j=2}^N \p^*A_j \subseteq 
\Big(\bigcup\limits_{j=3}^{N} Z_{2,\ldots,j-1;j}^{}\Big)\cup 
\p^*A_1\setminus \p^*\Big(\bigcup\limits_{j=2}^N A_j\Big).
\end{equation}
In view of \eqref{nteng2}, we have
$$
\p^*\Big(\bigcup\limits_{j=2}^N A_j\Big) = \p^*(\R^n\setminus A_1) = \p^*A_1,
$$
whence from \eqref{birlashmalar},
$$
\cH^{n-1}\Big(\p^*A_1 \setminus  \bigcup\limits_{j=2}^N \p^*A_j\Big)\le 
\sum\limits_{j=3}^N \cH^{n-1}(Z_{2,\ldots,j-1;j}^{})=0. 
$$
\end{proof}

\begin{remark}\label{rem:boundary_neighbor}
From Proposition \ref{teo:boundary_neighbor} it follows that 
%$$
%\Big\{x\in \p^* A_i:\,\, \text{$x\in \p^* A_j$ for 
%some $j\in \{1,\ldots, N\},$ $j\ne i$}\Big\} = 
%\bigcup\limits_{j=1,\,j\ne i}^N \p^*A_j\cap \p^* A_i.
%$$
%Moreover,
\begin{align*}
\Per(\cA,\Omega)= & \frac12\sum\limits_{j=1}^N \cH^{n-1}(\Omega\cap \p^*A_j)
= \sum\limits_{j=2}^N \sum\limits_{i=1}^{j-1}
\cH^{n-1}(\Omega\cap \p^*A_j\cap \p^*A_i).
\end{align*}
Since $\cH^{n-1}(\Omega\cap \p^*A_j\cap \p^*A_i)$ is the $(n-1)$-dimensional 
area of the interface
between the phases $A_i$ and $A_j,$ 
$\Per(\cA,\Omega)$ measures the total perimeter  of the
interfaces in $\Omega.$ 
\end{remark}

\subsection{$(\Lambda,r_0)$-minimizers}
In order to prove Theorem \ref{teo:density_est_ATW} it is 
convenient to give the following definition.

\begin{definition}[\bf $(\Lambda,r_0)$-minimizers]\label{def:almost_min}
Given $\Lambda\ge0$ and $r_0\in(0,+\infty]$
we say that a partition $\cA\in \P(N)$ is a $(\Lambda,r_0)$-minimizer 
of $\Per$  in $\R^n$ (a $(\Lambda,r_0)$-minimizer, for short)
if 
$$
\Per(\cA,B_r(x)) \le \Per(\cB,B_r(x)) +\Lambda |\cA\Delta\cB|
$$
whenever $x\in\R^n,$ $\cB\in \P(N),$ 
$\cA\Delta\cB\strictlyincluded B_r(x),$ and $r\in(0,r_0).$
\end{definition}

The crucial technical tool is the following.

\begin{theorem}[\bf Density estimates for $(\Lambda,r_0)$-
minimizers]\label{teo:density_est}
Let  $\cA\in \P(N)$ be a $(\Lambda,r_0)$-minimizer, 
$i\in\{1,\ldots,N\}$ and 
$\hat r_0:= \min\{r_0,\frac{n}{4(N-1)\Lambda}\} $ if $\Lambda>0$
and $\hat r_0:= r_0$ if $\Lambda=0.$
Then for any $x\in   \p A_i$  and $r\in (0, \hat r_0)$ 
the following density estimates hold:
\begin{equation}\label{volume_density}
\Big(\dfrac{1}{2N}\Big)^n \le \dfrac{|A_i\cap B_r(x)|}{|B_r(x)|}
\le 1- \dfrac{1}{2^n}\Big(1- \dfrac{1}{2(N-1)}\Big)^n,
\end{equation}
\begin{equation}\label{perimeter_density}
c_{n,N} \le \dfrac{P(A_i,B_r(x))}{r^{n-1}} \le  \frac{2N-1}{2(N-1)}\,n\omega_n,  
\end{equation}
where 
\begin{equation}\label{maaqweod}
c_{n,N} := \frac{n\omega_n^{1/n}(2^{1/n} -1 )}{2^{n+1/n} N^{n-1} }. 
\end{equation}
Moreover, 
\begin{equation}\label{cover}
\sum\limits_{i=1}^N\cH^{n-1}(\p A_i\setminus \p^* A_i)=0. 
\end{equation}
\end{theorem}

\begin{proof}
We may suppose  $i=1.$  Moreover, 
since $ \overline{\p^*A_1}=\p A_1,$ it suffices to show
\eqref{volume_density}-\eqref{perimeter_density} whenever
$x\in \p^* A_1.$  Writing $B_\rho:=B_\rho(x)$ for $\rho>0,$ 
we will show that for a.e. $r\in (0,\hat r_0)$ one has
\begin{equation}\label{eeerrr}
P(\R^n\setminus A_1,B_r) \le \cH^{n-1}((\R^n\setminus A_1) \cap \p B_r) +2
\Lambda |(\R^n\setminus A_1) \cap B_r|. 
\end{equation}

Choose $r\in(0,\hat r_0)$ satisfying
\begin{equation}\label{eqqq111}
\sum\limits_{j=1}^N \cH^{n-1}(\p B_r\cap \p^* A_j)=0 
\end{equation}
and define the competitor $\cB\in \P(N)$ as
$$
\cB:=(A_1\cup B_r,A_2\setminus B_r,\ldots,A_N\setminus B_r).
$$ 
Then $\cA\Delta\cB\strictlyincluded B_s$ for every $s\in (r,\hat r_0)$ 
and thus, by $(\Lambda,r_0)$-minimality,
\begin{equation}\label{rytur}
\begin{aligned}
0\le & 2\Per(\cB,B_s) - 2\Per(\cA,B_s) +2\Lambda|\cA\Delta\cB| =  
P(A_1\cup B_r,B_s) - P(A_1,B_s)\\
& +\sum\limits_{j=2}^N \Big(P(A_j\setminus B_r,B_s) - P(A_j,B_s)\Big) 
+2\Lambda|B_r\setminus A_1| + 2\Lambda\sum\limits_{j=2}^N |A_j\cap B_r|.
\end{aligned} 
\end{equation}
By the  disjointness of the $A_j$\!'s we have 
\begin{equation}\label{aaeee}
\sum\limits_{j=2}^N |A_j\cap B_r| = |B_r\setminus A_1|. 
\end{equation}
Moreover, recalling that $A_j^{(1)}=A_j,$ 
from the relation \eqref{per_difference}, \eqref{eqqq111} and 
$\cH^{n-1}(B_s\cap \{\nu_{A_j}=-\nu_{B_r}\})=0,$
we get 
\begin{equation}\label{set_operation12}
P(A_j\setminus B_r,B_s) =  P(A_j,B_s\setminus \overline{B_r}) +
\cH^{n-1}(A_j \cap \p B_r)\qquad\forall j\in \{2,\ldots,N\}.
\end{equation}
Thus, 
\begin{equation*}%\label{ttttt}
\sum\limits_{j=2}^N P(A_j\setminus B_r,B_s) = \sum\limits_{j=2}^N
P(A_j,B_s\setminus \overline{B_r}) + \sum\limits_{j=2}^N
\cH^{n-1}(A_j \cap \p B_r).
\end{equation*}
By the  disjointness of the $A_j$\!'s, Theorem \ref{cor:support_of_Dchi}
and the choice of $r$ in \eqref{eqqq111}, 
$$
\sum\limits_{j=2}^N \cH^{n-1}(A_j \cap \p B_r) = 
\cH^{n-1}(A_1^{(0)}\cap \p B_r)=\cH^{n-1}((\R^n\setminus A_1) \cap \p B_r).
$$
Therefore,
\begin{equation}\label{aee}
\sum\limits_{j=2}^N P(A_j\setminus B_r,B_s) = \sum\limits_{j=2}^N
P(A_j,B_s\setminus \overline{B_r}) + 
\cH^{n-1}((\R^n\setminus A_1) \cap \p B_r). 
\end{equation}
Finally, since $\cH^{n-1}(B_s\cap \{\nu_{A_1} = \nu_{B_r}\})=0$
by \eqref{eqqq111}, from \eqref{per_union} we deduce
\begin{equation}\label{aaaeee}
P(A_1\cup B_r,B_s) = P(A_1, B_s\setminus \overline{B_r}) +
\cH^{n-1}((\R^n\setminus A_1)\cap \p B_r). 
\end{equation}
Now inserting \eqref{aaeee}, \eqref{aee} and \eqref{aaaeee}  in \eqref{rytur}
we get 
\begin{equation}\label{rtzz}
P(A_1,B_r) + \sum\limits_{j=2}^N P(A_j,B_r) \le 
2\cH^{n-1}((\R^n\setminus A_1) \cap \p B_r) +
4\Lambda |(\R^n\setminus A_1) \cap B_r|. 
\end{equation}
Applying \eqref{famfor} and using the  disjointness
 of the  $A_j$\!'s we get 
$$
\sum\limits_{j=2}^N P(A_j,B_r) \ge  P\Big(\bigcup\limits_{j=2}^N A_j,B_r\Big) =
P(\R^n\setminus A_1,B_r) =P(A_1,B_r)
$$
and thus from \eqref{rtzz} we establish \eqref{eeerrr}.

Adding $\cH^{n-1}((\R^n\setminus A_1) \cap \p B_r)$ to 
both sides of \eqref{eeerrr}
and using \eqref{eqqq111} we get 
\begin{equation}\label{epsilon_delta}
P((\R^n\setminus A_1)\cap B_r) \le 
2 \cH^{n-1}((\R^n\setminus A_1) \cap \p B_r) +2
\Lambda |(\R^n\setminus A_1) \cap B_r|. 
\end{equation}
Now by the isoperimetric inequality \cite{DG:58-1},
\begin{equation}\label{diff.eq}
n\omega_{n}^{1/n}|(\R^n\setminus A_1)\cap B_r|^{\frac{n-1}{n}} \le 2 
\cH^{n-1}((\R^n\setminus A_1) \cap \p B_r) +2
\Lambda |(\R^n\setminus A_1) \cap B_r|. 
\end{equation}
Since $r<\hat r_0\le \frac{n}{4(N-1)\Lambda},$  
$$
2\Lambda |(\R^n\setminus A_1) \cap B_r|^{\frac{1}{n}}\le 
2\Lambda\omega_n^{1/n}\hat r_0
\le \frac{n\omega_n^{1/n}}{2(N-1)}.
$$
As a result, from \eqref{diff.eq} we obtain 
\begin{equation}\label{dum_dum_mast_h}
\dfrac12 \Big(1 -  \dfrac{1}{2(N-1)}\Big) n\omega_n^{1/n} 
|(\R^n\setminus A_1) \cap B_r|^{\frac{n-1}{n}} \le 
\cH^{n-1}((\R^n\setminus A_1) \cap \p B_r). 
\end{equation}
Set $m(\rho):=|(\R^n\setminus A_1) \cap B_\rho|,$ $\rho>0.$
Since $x\in \p A_1,$ one has $m(\rho)>0$ for any $\rho>0$
and by the coarea formula  (see e.g.  \cite[Example 13.4]{Mag12}) 
$m(\cdot)$ is absolutely continuous and 
$m'(\rho):=\cH^{n-1}((\R^n\setminus A_1) \cap \p B_\rho)$
for a.e. $\rho>0.$  Now by \eqref{dum_dum_mast_h} 
$$
\dfrac12 \Big(1 -  \dfrac{1}{2(N-1)}\Big) n\omega_n^{1/n} 
m(r)^{\frac{n-1}{n}} \le 
m'(r),\text{for a.e. $r\in (0,\hat r_0)$}. 
$$
Integrating this differential inequality we get
$$
|(\R^n\setminus A_1) \cap B_r| \ge 
\dfrac{1}{2^n} \Big(1-  \dfrac{1}{2(N-1)}\Big)^n  \omega_nr^n,
$$
i.e. 
$$
\dfrac{|A_1\cap B_r|}{|B_r|}
\le 1- \dfrac{1}{2^n}\Big(1- \dfrac{1}{2(N-1)}\Big)^n,
$$ 
which is the upper volume density estimate in \eqref{volume_density}.

Since $2\Lambda r\le \frac{n}{2(N-1)},$ from \eqref{eeerrr}
we obtain also
$$
P(A_1,B_r) \le \cH^{n-1}(\p B_r) +2\Lambda|B_r| \le 
n\omega_nr^{n-1} + \frac{n\omega_n}{2(N-1)}r^{n-1}=
\frac{2N-1}{2(N-1)}\, n\omega_nr^{n-1} 
$$
for a.e. $r\in (0,\hat r_0).$ Now the left-continuity of $\rho\mapsto P(A_1,B_\rho)$
implies the upper perimeter density estimate in \eqref{perimeter_density}.

Let us prove the lower volume density estimate. As above we may suppose 
$i=1$ and take $x\in \p^*A_1.$ Writing $B_\rho:=B_\rho(x)$ for $\rho>0,$
we will show that for a.e. $r\in (0,\hat r_0)$ one has
\begin{equation}\label{delta_force}
P(A_1,B_r) \le (N-1)\cH^{n-1}(A_1\cap\p B_r) +2(N-1)\Lambda|A_1\cap B_r|.
\end{equation} 

Set 
$$
I:=\{j\in \{2,\ldots, N\}:\,\, \cH^{n-1}(B_{\hat r_0}
\cap \p^*A_1\cap \p^*A_j)>0\}.
$$
Since $x\in \p A_1,$ one has $I\ne\emptyset.$ 
Let $r\in (0,\hat r_0)$ satisfy \eqref{eqqq111}.
By virtue of Proposition \ref{teo:boundary_neighbor} and Remark 
\ref{rem:boundary_neighbor},
\begin{equation}\label{qushnilar}
P(A_1,B_r) \le 
\sum\limits_{j=2}^N \cH^{n-1}(B_r\cap \p^*A_1\cap \p^* A_j) 
=\sum\limits_{j\in I} \cH^{n-1}(B_r\cap \p^*A_1\cap \p^* A_j).
\end{equation}
For every $j\in I$ let us define the competitor  $\cB^{(j)}\in \P(N)$ as
$$
\cB^{(j)}:=(A_1\setminus B_r,A_2,\ldots, A_{j-1},A_j\cup (A_1\cap B_r),A_{j+1},\ldots, A_N).
$$
By the $(\Lambda,r_0)$-minimality of $\cA,$ for every $s\in (r,\hat r_0)$ one has
\begin{equation}\label{min_j_case}
P(A_1,B_s) + P(A_j,B_s) \le P(A_1\setminus B_r,B_s) + P(A_j\cup(A_1\cap B_r),B_s)
+4\Lambda |A_1\cap B_r|. 
\end{equation}
From \eqref{eqqq111} and \eqref{ess_intersection}
\begin{equation}\label{uni_ess}
\p^*(A_1\cap B_r) \approx (A_1\cap \p B_r) \cup (B_r\cap \p^*A_1). 
\end{equation}

Observe that 
\begin{equation}\label{sfwervade}
\cH^{n-1}(B_s\cap \{\nu_{A_j} = \nu_{A_1\cap B_r}\})=0 
\end{equation}
for any $j\in I.$  
Indeed, by \eqref{uni_ess} 
$$
B_s\cap \{\nu_{A_j} = \nu_{A_1\cap B_r}\} \approx  
(A_1\cap \{\nu_{A_j}\cap \nu_{B_r}\})\cup 
(B_r\cap \{\nu_{A_j} = \nu_{A_1}\}).
$$
By \eqref{eqqq111}, $\cH^{n-1}(A_1\cap \{\nu_{A_j}\cap \nu_{B_r}\}) = 0.$ 
On the other hand, since $A_j\cap A_1 =\emptyset,$ one has 
$\nu_{A_j}(x) = -\nu_{A_1}(x)$ for $\cH^{n-1}$-a.e. 
$x\in \p^* A_j\cap \p^* A_1,$
and hence $\cH^{n-1}(B_r\cap \{\nu_{A_j} = \nu_{A_1}\}) = 0.$

From \eqref{per_union} and \eqref{sfwervade} it follows that
\begin{equation}\label{pppopop}
\begin{aligned}
P(A_j\cup(A_1\cap B_r),B_s) = &
\cH^{n-1}((A_1\cap B_r)^{(0)}\cap B_s\cap \p^*A_j) \\
& +\cH^{n-1}(A_j^{(0)}\cap B_s\cap \p^*(A_1\cap B_r)). 
\end{aligned}
\end{equation}
By Theorem \ref{cor:support_of_Dchi} 
\begin{equation}\label{wfraaf}
\begin{aligned}
\cH^{n-1}(B_s\cap \p^* A_j) = & \cH^{n-1}(E^{(0)}\cap B_s\cap \p^* A_j)+
\cH^{n-1}(E^{(1)} \cap B_s\cap \p^* A_j)\\
& + \cH^{n-1}(B_s\cap\p^* E\cap \p^* A_j)
\end{aligned}
\end{equation}
for every $E\in BV_\loc(\R^n,\{0,1\}).$
Hence, applying \eqref{wfraaf} with $E=A_1\cap B_r=E^{(1)},$
in view of  
$\cH^{n-1}(A_1\cap B_r\cap \p^*A_j)=0,$
$\cH^{n-1}(A_1\cap \p B_r\cap \p^*A_j)=0$ (see \eqref{eqqq111})
and \eqref{uni_ess}  we have
\begin{align*}
\cH^{n-1}((A_1\cap  B_r)^{(0)} \cap B_s& \cap \p^*A_j)\\
=&
\cH^{n-1}(B_s\cap \p^*A_j) -\cH^{n-1}( B_s\cap \p^*(A_1\cap B_r) 
\cap \p^*A_j)\\ 
=&P(A_j,B_s) -\cH^{n-1}(B_r \cap \p^* A_1\cap \p^* A_j).
\end{align*}
Similarly, since $A_j\cap A_1=\emptyset$ and $A_j\cap \p^*A_1 =\emptyset$ 
we have $\cH^{n-1}(A_j\cap \p B_r\cap \p^*(A_1\cap B_r))=0$
for any $j\in I$ and hence
\begin{align*}
\cH^{n-1}(A_j^{(0)}  \cap B_s\, \cap\, & \p^*(A_1\cap B_r))\\
&=
\cH^{n-1}(B_s\cap \p^*(A_1\cap B_r)) -\cH^{n-1}(B_s \cap \p^*(A_1\cap B_r)\cap \p^*A_j) \\
&= \cH^{n-1}(A_1\cap \p B_r) + P(A_1,B_r)-\cH^{n-1}(B_r \cap \p^* A_1\cap \p^* A_j).
\end{align*}
Therefore, from \eqref{pppopop} we get
\begin{equation}\label{eq3191}
\begin{aligned}
P(A_j\cup(A_1\cap B_r),B_s)= & P(A_j,B_s)+\cH^{n-1}(A_1\cap \p B_r)\\
& + P(A_1,B_r) -2\cH^{n-1}(B_r \cap \p^* A_1\cap \p^* A_j).
\end{aligned}
\end{equation}
Inserting this 
and 
\begin{equation*} 
P(A_1\setminus B_r,B_s) =  P(A_1,B_s\setminus \overline{B_r}) +
\cH^{n-1}(A_1 \cap \p B_r) 
\end{equation*}
(whose proof is the same as \eqref{set_operation12})  
in \eqref{min_j_case}   and using 
\eqref{eqqq111} once more
we get
\begin{equation}\label{sdfgd}
\cH^{n-1}(B_r\cap \p^*A_1\cap \p^* A_j) \le 
\cH^{n-1} (A_1 \cap \p B_r) + 2\Lambda |A_1\cap B_r|. 
\end{equation}% 
Summing these inequalities in  $j\in I$ and using \eqref{qushnilar} 
and $|I|\le N-1,$
we obtain \eqref{delta_force}. 

Now adding 
$\cH^{n-1}(A_1\cap \p B_r)$ to both sides of \eqref{delta_force}
we get
\begin{equation}\label{rtrey}
P(A_1\cap B_r)\le N\cH^{n-1}(A_1\cap \p B_r) + 2(N-1)\Lambda |A_1\cap B_r|.   
\end{equation}
Since $2
(N-1)\Lambda |A_1\cap B_r|^{1/n} \le \frac{n\omega_n^{1/n}}{2}
$
for any $r\in (0,\hat r_0),$ from the isoperimetric inequality we get 
$$
\frac{1}{2N} \,n\omega_n^{1/n} |A_1\cap B_r|^{\frac{n-1}n} \le 
\cH^{n-1}(A_1\cap \p B_r).
$$
Now proceeding as in the proof of the upper 
volume density estimate we get
the lower volume density estimate:
$$
|A_1\cap B_r|\ge \left(\frac{1}{2N}\right)^n \omega_nr^n.
$$

Now we prove the lower perimeter density estimate in 
\eqref{perimeter_density}. 
Notice that  $N\ge2,$ therefore
$$
\frac{1}{2N} \le \dfrac{1}{2}\Big(1- \dfrac{1}{2(N-1)}\Big).
$$
Hence from
the volume density estimates \eqref{volume_density} 
and \cite[Theorem I]{FFL:2006} 
\begin{align*}
P(A_1,B_r) \ge & \frac{n\omega_n^{1/n}(2^{1/n} - 1)}{2^{1+1/n}}\,
\min\Big\{|B_r\cap A_1|^{\frac{n-1}{n}},
|B_r\setminus A_1|^{\frac{n-1}{n}}\Big\}\\
\ge& 
\frac{n\omega_n^{1/n}(2^{1/n} - 1)}{2^{1+1/n}}\,
\min\Big\{\frac{1}{2N}, \dfrac{1}{2}\Big(1- \dfrac{1}{2(N-1)}\Big)
\Big\}^{n-1}|B_r|^{\frac{n-1}{n}} 
=  c_{n,N}r^{n-1}.
\end{align*}

Finally, \eqref{cover} is a consequence of a standard covering argument.
\end{proof}

\begin{remark}\label{rem:strange_almost_min}
Let $\alpha_1,\alpha_2>\frac{n-1}{n},$ $\Lambda_1\ge0,$ $\Lambda_2>0,$
$r_0\in(0,+\infty].$  Suppose that $\cA\in\P(N)$ satisfies 
$$
\Per(\cA,B_r(x)) \le \Per(\cB,B_r(x)) + \Lambda_1 |\cA\Delta \cB|^{\alpha_1}
+\Lambda_2 |\cA\Delta \cB|^{\alpha_2}
$$ 
whenever $\cB\in\P(N),$ $\cA\Delta\cB \strictlyincluded B_r(x)$ and $r\in (0,r_0).$
Then, repeating the proof of Theorem \ref{teo:density_est}, 
one obtains that \eqref{epsilon_delta} and \eqref{rtrey} 
are replaced by 
\begin{align*}
P((\R^n\setminus A_1)\cap B_r) \le &2 \cH^{n-1}((\R^n\setminus A_1) \cap \p B_r) +
  2\Lambda_1 |(\R^n\setminus A_1) \cap B_r|^{\alpha_1}\\
  & + 2\Lambda_2 |(\R^n\setminus A_1) \cap B_r|^{\alpha_2}
\end{align*}
and
$$
P(A_1\cap B_r) \le N\cH^{n-1}(A_1\cap \p B_r) +
2(N-1)\Lambda_1|A_1\cap B_r|^{\alpha_1}+
2(N-1)\Lambda_2|A_1\cap B_r|^{\alpha_2}
$$
respectively. Thus, for every $i\in \{1,\ldots,N\},$  
for every $x\in \p A_i$ and for any $r\in (0,\tilde r_0),$ 
the relations \eqref{volume_density}-\eqref{cover}
hold, where 
$$
\tilde r_0 = 
\begin{cases}
\min\{r_0,\omega_n^{-1/n} 
\big(\frac{n\omega_n^{1/n}}{4(N-1)\Lambda_2}\big)^{\frac{1}{n\alpha_2-n+1}}\} &
\text{if $\Lambda_1=0$},\\ 
\min\{r_0, \omega_n^{-1/n} 
\big(\frac{n\omega_n^{1/n}}{8(N-1)\Lambda_1}\big)^{\frac{1}{n\alpha_1-n+1}},
\omega_n^{-1/n} 
\big(\frac{n\omega_n^{1/n}}{8(N-1)\Lambda_2}\big)^{\frac{1}{n\alpha_2-n+1}}\}  &
\text{if $\Lambda_1>0.$}    
\end{cases}
$$
This will be used in the proof 
of Theorem \ref{teo:existence_GMM_2}.
\end{remark}

From \eqref{cover} it follows that $\cH^{n-1}\res\p^* A_i = \cH^{n-1}\res\p A_i$
for every $i=1,\ldots,N.$

\begin{remark}\label{muhim_remark}
Let $x\in \R^n$ and let $B_r:=B_r(x),$ $r\in (0,\hat r_0)$ be any 
ball such that
$$
\sum\limits_{j=1}^N \cH^{n-1}(\p^*A_j\cap \p^* B_r) =0
$$
($x$ not necessarily lies on $\bigcup\limits_{j=1}^N \p A_j$).
Then comparing $\cA$ with 
$
\cB:=(A_1\cup B_r,A_2\setminus B_r,\ldots,A_N\setminus B_r)
$
as in the proof of Theorem \ref{teo:density_est} we get 
$$
P(A_1,B_r) \le \cH^{n-1}((\R^n\setminus A_1)\cap \p B_r) +\Lambda 
|B_r\cap (\R^n\setminus A_1)|
$$
and therefore 
\begin{equation}\label{uniform_per_b}
\dfrac{P(A_1,B_r(x))}{r^{n-1}} \le C(n,N),\qquad \forall r\in (0,\hat r_0).    
\end{equation}
By  symmetry \eqref{uniform_per_b} holds for every $i=1,\ldots,N.$
\end{remark}

\subsection{Bounded partitions}

The multiphase analog of a bounded phase in $\R^n$ is the following.

\begin{definition}[\bf Bounded partition]\label{def:g_partitions}
A  partition $\cC=(C_1,\ldots,C_{N+1})\in\P(N+1)$ is called 
{\bf bounded }  if $C_i$ is bounded for each $i=1,\ldots,N.$ 
\end{definition}

Therefore, $C_{N+1}$ is the only unbounded component of $\cC.$
We denote by $\P_b(N+1)$ the collection of all bounded partitions of $\R^n.$ 
%The partition $(\emptyset,\ldots,\emptyset,\R^n)$ is called 
%the trivial partition. 

Given $\cA\in \P_b(N+1),$ 
we denote by $$\hull{\cA}$$ the closed convex hull of  
$\bigcup\limits_{i=1}^N A_i.$ 
Since $\cA\Delta\cB \strictlyincluded \R^n$ 
for every $\cA,\cB\in \P_b(N+1),$ 
$$
|\cA\Delta\cB| = \sum\limits_{j=1}^{N+1} |A_j\Delta B_j|
$$
is the $L^1(\R^n)$-distance in $\P_b(N+1).$ 

The following compactness result  can be proven
similarly to Theorem \ref{teo:compactness}. 

\begin{theorem}[\bf Compactness]\label{prop:compactness}
Let $\{\cA^{(l)}\}\subset \P_b(N+1)$  
and $\Omega\in \K_b(\R^n)$ be such that
$$
\sup\limits_{l\ge1} \Per(\cA^{(l)})< +\infty,
\qquad \hull{\cA^{(l)}}\subseteq \Omega
\qquad \forall l\ge1. 
$$
Then there exist $\cA\in\P_b(N+1)$ and a subsequence
$\{\cA^{(l_k)}\}$ converging to $\cA$ in $L^1(\R^n)$ as $k\to+\infty.$ Moreover,
$
\bigcup\limits_{i=1}^N A_j\subseteq 
\overline \Omega. 
$
\end{theorem}

%%%%%%%%%%%%%%%%%%%%%%%%%%%%%%%%%%%%%%%%%%%%%%%%%%%%%%%%%%%%%%
\section{Existence of $GMM$} \label{sec:existence_GMM}
%%%%%%%%%%%%%%%%%%%%%%%%%%%%%%%%%%%%%%%%%%%%%%%%%%%%%%%%%%%%%%

Given $E,F\subseteq \R^n$ set
$$
\bar\sigma(E,F):= \int_{E\Delta F} \dist(x,\p F)dx.
$$
Note that $\bar\sigma(E,F)=0$ if $|E\Delta F|=0$ whereas $\bar\sigma(E,F)=+\infty$
if $\p F =\emptyset$ and $|E\Delta F|>0.$
Moreover, if  $X,Y\subseteq \R^n$ are measurable 
and $\p Y\ne\emptyset,$
\begin{equation}\label{eq:pro_distance}
\begin{aligned}
\int_{X\Delta Y } \dist(x,\p Y) dx = \int_X \sdist(x,\p Y) dx - 
\int_Y\sdist(x,\p Y) dx\quad \text{if $X\cap Y$ is bounded},\\
\int_{X\Delta Y } \dist(x,\p Y) dx = \int_{Y^c} \sdist(x,\p Y) dx - 
\int_{X^c}\sdist(x,\p Y) dx\quad \text{if $X^c\cap Y^c$ is bounded}.
\end{aligned}
\end{equation}
Now  the {\it nonsymmetric distance} between $\cA,\cB\in  \P_b(N+1)$ is 
defined as 
$$
\sigma(\cA,\cB):= \sum\limits_{i=1}^{N+1} \bar\sigma(A_i,B_i),
$$
where $N+1\ge2.$
Observe that for every $\cB\in\P_b(N+1)$ the map 
$\sigma(\cdot,\cB)$ is $L^1(\R^n)$-lower semicontinuous.
\smallskip

\begin{definition}[\bf The functional $\func $]
We let $\func: \P_b(N+1)\times \P_b(N+1)\times[1,+\infty)\to[0,+\infty]$ 
be the functional
 defined as 
$$
\func(\cB,\cA;\lambda) = \Per(\cB) +\frac{ \lambda}{2}\,\sigma(\cB,\cA)
=\frac12 \sum\limits_{j=1}^{N+1} P(B_j) + 
\frac{\lambda}{2} \sum\limits_{j=1}^{N+1} 
\int_{B_j\Delta A_j}\dist(x,\p A_j)dx.
$$
 
\end{definition}
 
The domain of $\func $ is independent of $\Z,$ and  $\func $ 
is the natural generalization of the Almgren-Taylor-Wang functional \cite{ATW93}
to the case of partitions \cite{Car:Th,DG-96}.
One can readily check that the map $\cB\in \P_b(N+1)\mapsto \func(\cB,\cA;\lambda)$
is $L^1(\R^n)$-lower semicontinuous.   

\begin{theorem}[\bf Existence of minimizers of $\func $]\label{teo:existence_of_minimizers}
Given $\cA\in \P_b(N+1)$ and $\lambda\ge1$ the problem 
\begin{equation}\label{min.problem}
\inf\limits_{\cB\in \P_b(N+1)} \func(\cB,\cA;\lambda) 
\end{equation}
has a solution. Moreover, every minimizer $\cA(\lambda)=(A_1(\lambda),\ldots,
A_{N+1}(\lambda))$ 
satisfies the bound 
\begin{equation*}%\label{uniform_bound}
\bigcup\limits_{i=1}^N A_i(\lambda)\subseteq \hull{\cA}.
\end{equation*}
\end{theorem}

\begin{proof}
Given a partition $\cB\in \P_b(N+1)$ define the competitor  $\cB'\in \P_b(N+1)$ as
\begin{equation}\label{cutting}
\cB':= \Big(B_1\cap \hull{\cA},\ldots, B_N\cap\hull{\cA}, \R^n\setminus
\bigcup\limits_{i=1}^N ( B_i \cap \hull{\cA})\Big ). 
\end{equation} 
Since $\hull{\cA}$ is convex and closed, by the comparison theorem of
\cite[page 152]{LAln} we have 
$P(B_i)\ge P(B_i\cap \hull{\cA})$ for  $i=1,\ldots,N,$
and
\begin{align*}
P(B_{N+1}) =  & P\Big(\bigcup\limits_{i=1}^N B_i\Big) \ge 
P\Big(\Big(\bigcup\limits_{i=1}^N 
B_i\Big)\cap \hull {\cA}\Big)\\
= & P\Big(\bigcup\limits_{i=1}^N (B_i\cap \hull{\cA})\Big) = 
P\Big(\R^{n}\setminus \bigcup\limits_{i=1}^N (B_i\cap \hull{\cA})\Big),
\end{align*}
with equality if and only if 
$|\bigcup\limits_{i=1}^N B_i\setminus \hull{\cA}|=0.$ 
In addition, for $i=1,\ldots,N,$ 
\begin{equation}\label{eq:df}
\begin{aligned}
\int_{B_i\Delta A_i} \dist(x,\p A_i) dx = &
\int_{B_i\setminus A_i} \dist(x,\p A_i) dx 
+\int_{A_i\setminus B_i} \dist(x,\p A_i) dx\\
\ge &
\int_{(B_i\cap \hull{\cA})\setminus A_i} 
\dist(x,\p A_i) dx+\int_{A_i\setminus (B_i\cap \hull{\cA})}
\dist(x,\p A_i) dx\\
=&
\int_{(B_i\cap \hull{\cA})\Delta A_i} \dist(x,\p A_i) dx,
\end{aligned} 
\end{equation}
where we used the nonnegativity of the distance function  and 
$A_i\setminus B_i=A_i\setminus (B_i\cap \hull{\cA}).$
The equality in \eqref{eq:df} holds if and only if 
$\Big|\bigcup\limits_{i=1}^N B_i\setminus \hull{\cA}\Big|=0.$
For the same reason,  since $A_{N+1}^c = \bigcup\limits_{i=1}^N 
A_i \subseteq \hull{\cA},$
\begin{align*}
\int_{B_{N+1}\Delta A_{N+1}} \dist(x,\p A_{N+1}) dx = & 
\int_{B_{N+1}^c\Delta A_{N+1}^c} \dist(x,\p A_{N+1}) dx\\
\ge & \int_{(B_{N+1}^c\cap \hull{\cA})\Delta A_{N+1}^c} \dist(x,\p A_{N+1}) dx. 
\end{align*}
So we have 
\begin{equation*}%\label{eq:inclusion_to_Ball}
\func(\cB,\cA;\lambda) \ge \func(\cB',\cA;\lambda) \qquad \forall \cB\in \P_b(N+1)
\end{equation*}
and the inequality is strict whenever 
$\Big|\bigcup\limits_{i=1}^N B_i\setminus \hull{\cA}\Big|>0.$

Let $\{\cB^{(k)}\}\subseteq \P_b(N+1)$ be a minimizing sequence, which 
can be supposed so that $\hull{\cB^{(k)}} \subseteq \hull{\cA}$
and $\func(\cB^{(k)},\cA;\lambda) \le \func(\cT,\cA;\lambda),$ 
$\cT:=(\emptyset,\ldots,\emptyset,\R^n)$ being the trivial partition, so that
$$
\Per(\cB^{(k)})\le \frac{\lambda}{2} \,\sigma(\cT,\cA) = 
\frac{\lambda}{2} \sum\limits_{j=1}^N \int_{A_j}\Big(\dist(x,\p A_j) + 
\dist(x,\p A_{N+1})\Big)\,dx \qquad \forall k\ge1.
$$
By Theorem \ref{prop:compactness} there exists $\cA(\lambda)\in \P_b(N+1)$
such that (passing to a not relabelled subsequence) 
$\cB^{(k)}\to\cA(\lambda)$ in $L^1(\R^n)$ as $k\to+\infty.$
Then the $L^1(\R^n)$-lower semicontinuity
of $\func(\cdot,\cA;\lambda)$
implies that $\cA(\lambda)$  is a solution to \eqref{min.problem}.

Now let $\cA(\lambda)$ be a  minimizer of $\func(\cdot,\cA;\lambda).$ 
If $\big|\bigcup\limits_{j=1}^N A_j(\lambda)\setminus \hull{\cA}\big|>0$ 
then, as shown above,
$\func(\cA(\lambda),\cA;\lambda)>\func(\cA(\lambda)',\cA;\lambda),$ 
where $\cA(\lambda)'$ is defined as 
in \eqref{cutting}, which contradicts the minimality of $\cA(\lambda).$
\end{proof}

\begin{remark}\label{rem:cut_convex}
Let $C\subseteq\R^n$ be a compact convex set. Suppose that 
$\cG\in\P_b(N+1)$ satisfies $\bigcup\limits_{j=1}^N G_j \subseteq C;$
from Theorem \ref{teo:existence_of_minimizers} it follows that
every minimizer $\cA(\lambda)\in \P_b(N+1)$ of $\func(\cdot,\cG;\lambda)$ 
satisfies $\hull{\cA(\lambda)} \subseteq C.$
This  gives an a priori bound
for minimizers of $\func(\cdot,\cG;\lambda)$ just from a
bound for the initial partition and will be 
used in the  proofs of Theorems \ref{teo:existence_GMM} and 
\ref{teo:existence_GMM_2}.
\end{remark}

\begin{remark}
Suppose that $\cG\in\P_b(N+1)$ and  $G_i =\emptyset$ for some 
$i\in \{1,\ldots,N\}.$ Then by definition of $\bar\sigma$ 
every minimizer $\cA(\lambda)\in \P_b(N+1)$ of $\func(\cdot,\cG;\lambda)$ 
satisfies $A_i(\lambda)=\emptyset.$ In particular, for 
$\cG=(G,\emptyset,\ldots,\emptyset,\R^n\setminus G),$
the $GMM$ problem for $\func(\cdot,\cG;\lambda)$ agrees with the 
$GMM$ problem for  the Almgren-Taylor-Wang functional 
\begin{equation}\label{eq:standard_ATW}
E\in BV(\R^n) \mapsto \atw(E,G;\lambda):=P(E)+  
\lambda\int_{E\Delta G} \dist(x,\p G)dx. 
\end{equation}

\end{remark}

\begin{proposition}[\bf Behaviour of $\cA(\lambda)$ as time
goes to $0$]\label{asymptota}
Let $\cA\in \P_b(N+1)$ be such that $\sum\limits_{j=1}^{N+1} 
|\overline{A_j}\setminus A_j|=0,$ and $\cA(\lambda)$ 
be a minimizer of $\func(\cdot,\cA;\lambda).$ 
Then:
\begin{itemize}
\item[a)]  $\lim\limits_{\lambda\to+\infty} |\cA(\lambda)\Delta \cA|=0,$
\item[b)]  $\lim\limits_{\lambda\to+\infty} \Per(\cA(\lambda))=\Per(\cA),$
\item[c)]  $\lim\limits_{\lambda\to+\infty} \lambda\sigma(\cA(\lambda), \cA)=0.$
\end{itemize}
\end{proposition}

\begin{proof}
a) Choose any sequence $\lambda_k\to+\infty.$
Since $\func(\cA(\lambda_k),\cA;\lambda_k)\le \func(\cA,\cA;\lambda_k)=\Per(\cA),$
we have $\Per(\cA(\lambda))\le \Per(\cA)$ and 
\begin{equation}\label{dist_teng_0}
 \lim\limits_{k\to+\infty} \sigma(\cA(\lambda_k),\cA) =0.
\end{equation}
Moreover, by Theorem \ref{teo:existence_of_minimizers} 
$\hull{\cA(\lambda)}\subseteq\hull{\cA},$ therefore Proposition
\ref{prop:compactness} yields the existence of a subsequence 
$\{\lambda_{k_l}\}_l$ and of $\cB\in \P_b(N+1)$ such that 
$\cA(\lambda_{k_l})\to \cB$ in $L^1(\R^n)$ as $l\to+\infty.$
Now the lower semicontinuity of $\sigma(\cdot,\cA)$ and \eqref{dist_teng_0}
imply  $\sigma(\cB,\cA)=0.$ Then from the assumption on $\cA$ we get
$\cA=\cB.$ Since $\lambda_k$ is arbitrary, a) follows. 

b) Since $\Per(\cA(\lambda))\le \Per(\cA),$ from a) we obtain
$$
\Per(\cA)\le \liminf\limits_{\lambda\to+\infty} \Per(\cA(\lambda))
\le \limsup\limits_{\lambda\to+\infty} \Per(\cA(\lambda))\le\Per(\cA).
$$

c) From b) we have 
$$
\limsup\limits_{\lambda\to+\infty} \lambda\sigma(\cA(\lambda),\cA) 
\le 2 \limsup\limits_{\lambda\to+\infty} (\Per(\cA) - \Per(\cA(\lambda)) )= 0.
$$
\end{proof}

\begin{theorem}[\bf Density estimates]\label{teo:density_est_ATW}
Suppose that $\cA\in \P_b(N+1)$ and let $\cA(\lambda)\in \P_b(N+1)$ 
be a minimizer of 
$\func(\cdot, \cA;\lambda).$ Then for every $i\in \{1,\ldots,N+1\}$  
\begin{equation}\label{eq:vol.density_est}
\Big(\dfrac{1}{2(N+1)}\Big)^n \le \dfrac{|A_i(\lambda)\cap B_r(x)|}{|B_r(x)|}
\le 1- \dfrac{1}{2^n}\Big(1- \dfrac{1}{2N}\Big)^n,
\end{equation}
\begin{equation}\label{eq:per.density_est}
c_{n,N+1} \le \frac{P(A_i(\lambda),B_r(x))}{r^{n-1}} \le 
\dfrac{2N+1}{2N}\,n\omega_n  
\end{equation}
for any $x\in \p A_i(\lambda)$ and $r\in (0,\min\{1,\frac{n}{2\lambda 
N(\diam \hull{\cA}+2)}\}),$ where
$c_{n,N+1}$ is defined in \eqref{maaqweod} (with $N+1$ in place of $N$).
Moreover
\begin{equation*}%\label{federer}
\sum\limits_{j=1}^{N+1} \cH^{n-1}(\p A_j(\lambda)\setminus \p^*A_j(\lambda)) =0. 
\end{equation*}
\end{theorem}

\begin{proof} 
Without loss of generality, we may suppose $\p A_i \ne\emptyset$ for every 
$i=1,\ldots,N+1.$
Fix $r_0>0.$ Then for every $x\in \R^n$
and $\cC\in\P_b(N+1)$ such that $\cC\Delta \cA(\lambda)
\strictlyincluded B_\rho(x)$ with
$\rho\in (0,r_0),$ by Theorem \ref{teo:existence_of_minimizers} one has
$$
\dist(z,\p A_i) \le \diam\hull{\cA} +2\rho\qquad \forall 
i=1,\ldots,N+1,\,\,z\in \cC\Delta \cA(\lambda).
$$
Therefore the minimality of $\cA(\lambda)$ implies 
$$
\Per(\cA(\lambda),B_\rho(x)) \le \Per(\cC,B_\rho(x)) + 
\frac{\lambda}{2}\,\big(\diam\hull{\cA} +2r_0\big)|\cC\Delta\cA(\lambda)|,
$$
i.e. 
$$
\text{$\cA(\lambda)$ is a $(\Lambda,r_0)$-minimizer with 
$\Lambda= \frac{\lambda}{2}\,\big(\diam\hull{\cA} +2r_0\big).$}
$$
Now application of  Theorem \ref{teo:density_est} to $\cA(\lambda)$
with  $r_0=1$ finishes the proof.
\end{proof}

\begin{remark}\label{rem:interior_reg}
The density estimates show that the components of $\cA(\lambda)$ are
Lebes\-gue-equivalent to open sets. Indeed, since 
using  $\overline{E}\setminus E\subset \p E,$ and 
$\overline{E}\setminus \mathring{\overline{E}}\subset \p E$ 
($\mathring G$ being the interior of $G$),
we have 
$$
\sum\limits_{j=1}^{N+1} |A_j(\lambda)\Delta \mathring{ \overline{A_j(\lambda)}}| 
\le 
\sum\limits_{j=1}^{N+1} |\overline{A_j(\lambda)} \setminus A_j(\lambda)| +  
\sum\limits_{j=1}^{N+1} |\overline{A_j(\lambda)} \setminus 
\mathring{\overline{A_j(\lambda)}}|\le 
2\sum\limits_{j=1}^{N+1} |\p A_j(\lambda)|.
$$
Now by the density estimates $\sum\limits_{j=1}^{N+1} |\p A_j(\lambda)|=0,$
and therefore  $\sum\limits_{j=1}^{N+1} |A_j(\lambda)\Delta 
\mathring{ \overline{A_j(\lambda)}}|=0.$
\end{remark}

 To prove the existence of $GMM,$ we need the following 
corollary of Theorem \ref{teo:density_est_ATW}.

\begin{corollary}\label{cor:ATW}
Let $\epsilon>0$ and suppose that 
the components of $\cA\in \P_b(N+1)$ satisfy the density estimates 
\eqref{eq:vol.density_est}-\eqref{eq:per.density_est} for all 
$r\in (0,\epsilon].$
Then for every minimizer $\cA(\lambda)$ of $\func(\cdot,\cA;\lambda)$
in $\P_b(N+1)$ one has 
\begin{equation}\label{jami_had}
|\cA(\lambda)\Delta \cA| \le  
\frac{5^n\omega_n}{c_{n,N+1}}\,\left(\frac\ell\epsilon\right)^{n-1} \,  
\Per(\cA)\,\ell +
\frac{1}{\ell}\,\sigma(\cA(\lambda),\cA),\qquad \ell\ge \epsilon.
\end{equation}
\end{corollary}

\begin{proof}
Fix $\ell\ge \epsilon$ and $i\in \{1,\ldots,N+1\}$ and 
set
$$
E:=\{x\in A_i(\lambda)\Delta A_i:\,\, \dist(x,\p A_i)\ge \ell\},\qquad
F:=\{x\in A_i(\lambda)\Delta A_i:\,\, \dist(x,\p A_i)< \ell\}.
$$
By the Chebyshev inequality,
$$
|E|\le \frac{1}{\ell}\, \int_E \dist(x,\p A_i)dx\le 
\frac{1}{\ell}\,\int_{A_i(\lambda)\Delta A_i} \dist(x,\p A_i)dx.
$$
We cover the set $F$ with a family 
$\{\overline{B_{\ell}(x)}:\,\,x\in \p A_i\}$ 
of closed balls. By the Vitali lemma, there exists  
a finite subset $\{\overline{B_{\ell}(x_j)}\}_j$ of the covering, 
consisting of  disjoint  balls, such that 
$F \subset \bigcup\limits_{j} \overline{B_{5\ell}(x_j)}.$ Since 
by assumption
$A_i$ satisfies the lower perimeter density estimate in 
\eqref{eq:per.density_est} with $r=\epsilon,$ 
\begin{align*}
|F| \le \sum\limits_j 5^n\omega_n\ell^n \le 
\frac{5^n\omega_n}{c_{n,N+1}}\,\left(\frac\ell\epsilon\right)^n \,
\epsilon\,\sum\limits_j P(A_i,B_\epsilon(x_j)) \le 
\frac{5^n\omega_n}{c_{n,N+1}}\,
\left(\frac\ell\epsilon\right)^{n-1} \,  P(A_i)\,\ell.
\end{align*}
Thus, 
\begin{equation}\label{har_bir_had}
|A_i(\lambda)\Delta A_i| \le |E|+|F| \le 
\frac{1}{\ell}\,\int_{A_i(\lambda)\Delta A_i} \dist(x,\p A_i)dx 
+ \frac{5^n\omega_n}{c_{n,N+1}}\,\left(\frac\ell\epsilon\right)^{n-1} \,  P(A_i)\,\ell.
\end{equation}
Now \eqref{jami_had} follows  summing \eqref{har_bir_had} 
with respect to $i.$
\end{proof}

One of the main results of the present paper reads as follows.

\begin{theorem}[\bf Existence of $GMM$]\label{teo:existence_GMM}
Let  $\cG\in \P_b(N+1).$ Then $GMM(\func,\cG)$ is non empty. Moreover, 
there exists a constant $\widehat c=\widehat c(N,n,\cG)>0$ such that 
for any $\cM\in GMM(\func,\cG),$ 
\begin{equation}\label{hulder_est}
|\cM(t)\Delta \cM(t')|\le  \widehat c \,|t-t'|^{\frac{1}{n+1}}\qquad
\forall t,t'>0,\,\,|t-t'|<1
\end{equation}
and
\begin{equation}\label{dddid}
\bigcup\limits_{j=1}^N M_j(t) \subseteq \hull{\cG}
\qquad \forall t\ge0.
\end{equation}
In addition, if $\sum\limits_{j=1}^{N+1} 
|\overline{G_j}\setminus G_j| = 0,$ then \eqref{hulder_est} holds for any 
$t,t'\ge0$ and $|t-t'|<1.$
\end{theorem}

\begin{proof}
Set $2R:=\diam\hull{\cG}.$
Let $\cL(\lambda,k)=(L_1(\lambda,k),\ldots,L_{N+1}(\lambda,k)),$ 
$\lambda\ge1,$ $k\in\N_0$ be defined as follows:
$\cL(\lambda,0):=\cG,$ and for $k \geq 1$
$$
\func(\cL(\lambda,k),\cL(\lambda,k-1);\lambda) =\min\limits_{\cA\in \P_b(N+1)} 
\func(\cA,\cL(\lambda,k-1);\lambda);
$$
recall that the existence of minimizers follows from 
Theorem \ref{teo:existence_of_minimizers} and also
\begin{equation}\label{eq:unif_bounded}
\bigcup\limits_{j=1}^N  L_j(\lambda,k) \subseteq
\hull{\cG}\qquad \forall\lambda\ge1, \,\, k\in\N_0. 
\end{equation}
Clearly, 
$\func(\cL(\lambda,k),\cL(\lambda,k-1);\lambda)\le 
\func(\cL(\lambda,k-1),\cL(\lambda,k-1);\lambda),$
hence 
\begin{equation}\label{superstar}
\lambda \sigma(\cL(\lambda,k), \cL(\lambda,k-1) ) \le 
2\big(\Per(\cL(\lambda,k-1)) - \Per(\cL(\lambda,k))\big) \qquad \forall k\ge1. 
\end{equation}
Therefore, the sequence $k\in  \N_0\mapsto \Per(\cL(\lambda,k))$ 
is nonincreasing, and 
$\Per(\cL(\lambda,k)) \le \Per(\cG)$
for all $k\in\N_0$  and $\lambda\ge1,$ since $\cL(\lambda,0)=\cG.$

For every $t,t'>0,$ $0<t-t'<1$ let us prove 
\begin{align}\label{eq:muhim_holder}
|\cL(\lambda,[\lambda t])\Delta \cL(\lambda,[\lambda t'])|  
\le \widehat{c} |t-t'|^{\frac{1}{n+1}} 
\end{align}
provided that $\lambda$ is sufficiently large depending on $|t-t'|,$
$n,$ $N$ and $R,$ 
where 
$$
\widehat c: = \widehat c(N,n,\cG)= \left(\frac{5^n\omega_n}{c_{n,N+1}}\, 
\frac{n}{2N(R+1)}  + \frac{8N(R+1)}{n}\right)\,\Per(\cG ). 
$$

Set
$k_0:=[\lambda t'],$ $m_0:= [\lambda t].$ 
Let $\lambda\ge \max\{\frac{n}{4(R + 1)N},\frac{1}{|t-t'|}\}$ be so large 
that  $m_0\ge k_0+3\ge4$  and 
$\frac{n}{4\lambda N(R+1)|t-t'|^\alpha}<1,$ $\alpha:=\frac{1}{n+1}.$ 
Since each $\cL(\lambda,k),$ $k\ge1,$ satisfies the density 
estimates \eqref{eq:vol.density_est}-\eqref{eq:per.density_est} 
(Theorem \ref{teo:density_est_ATW}) for 
$r\in (0,\frac{n}{4\lambda N(R+1)}),$ 
we may apply Corollary \ref{cor:ATW} with 
$\ell = \frac{n}{4\lambda N(R+1)|t-t'|^\alpha}$
and $\epsilon=\frac{n}{4\lambda N(R+1)},$
the inequality $\Per(\cL(\lambda,k))\le \Per(\cG)$
and \eqref{superstar} to get 
\begin{align*}
|\cL(\lambda,m_0) \Delta \cL(\lambda,k_0)|\le  & 
\sum\limits_{k=k_0+1}^{m_0} |\cL(\lambda,k)\Delta \cL(\lambda,k-1)| \\
\le & \sum\limits_{k=k_0+1}^{m_0}
\frac{5^n\omega_n}{c_{n,N+1}}\, \frac{n}{4\lambda N(R+1)}\,|t-t'|^{-n\alpha}
\,\Per(\cL(\lambda,k-1))\\ 
  & + \frac{4N(R+1)}{n}\,\lambda|t-t'|^{\alpha}\, 
  \sigma(\cL(\lambda,k),\cL(\lambda,k-1))\\ 
\le &  \frac{5^n\omega_n}{c_{n,N+1}}\, 
\frac{n}{4N(R+1)}\,\Per(\cG)\,|t-t'|^{-n\alpha}\,\frac{m_0-k_0}{\lambda} \\
& + \frac{8N(R+1)}{n}\,|t-t'|^{\alpha}\,\sum\limits_{k=k_0+1}^{m_0} 
(\Per(\cL(\lambda,k-1)) - \Per(\cL(\lambda,k))).
\end{align*}
Since 
\begin{equation}\label{super_star_a}
m_0-k_0\le \lambda|t-t'|+1\le 2\lambda|t-t'|, 
\end{equation}
from the choice of $\alpha$ and  
the bound $\Per(\cL(\lambda,k))\le \Per(\cG),$
we establish
\begin{align*}
|\cL(\lambda,m_0) \Delta \cL(\lambda,k_0)|\le   &
\left(\frac{5^n\omega_n}{c_{n,N+1}}\, 
\frac{n}{2N(R+1)} 
 + \frac{8N(R+1)}{n}\right)\,\Per(\cG)\,|t-t'|^{\frac{1}{n+1}},
\end{align*}
which is \eqref{eq:muhim_holder}.

Now we prove the assertions of the theorem.  
Using the inclusion \eqref{eq:unif_bounded}, the inequality 
$\Per(\cL(\lambda,k))\le \Per(\cG),$
Proposition \ref{prop:compactness} and a diagonal argument we obtain
the existence of a diverging sequence $\{\lambda_h\}$ and 
$\cM(t)\in \P_b(N+1)$ such that 
\begin{equation}\label{l_one_yaqin}
\lim\limits_{h\to+\infty} |\cL(\lambda_h,[\lambda_ht])\Delta \cM(t)|=0 
\end{equation}
for every rational $t>0$ and also \eqref{dddid} holds.
By \eqref{eq:muhim_holder} $\cM(t)$ satisfies 
\begin{equation*}
|\cM(t)\Delta \cM(t')| \le \widehat c\,|t-t'|^{\frac{1}{n+1}} \qquad
\forall t',t\in \Q\cap (0,+\infty),\,\,|t-t'|<1. 
\end{equation*}
Hence this map extends uniquely to a map $\{\cM(t):\,t>0\}\subseteq\P_b(N+1)$ 
satisfying \eqref{hulder_est} and \eqref{dddid}.

To show that $\cM\in GMM(\func,\cG)$ it remains 
only to prove \eqref{l_one_yaqin} for any $t\ge0.$
Case $t=0$ is trivial: $\cM(0)=\cG.$ 
Fix $t>0.$ 
For every $\epsilon\in (0,1)$
take $t_\epsilon\in\Q\cap (0,+\infty)$ 
such that $|t-t_\epsilon|<\epsilon^{n+1}$ (recall that 
\eqref{l_one_yaqin} holds with $t_\epsilon$). 
Since $\cM$ satisfies \eqref{hulder_est}, from \eqref{eq:muhim_holder}
and \eqref{l_one_yaqin} (applied with $t_\epsilon$) we deduce 
\begin{align*}
\limsup\limits_{h\to+\infty} |\cL(\lambda_h,[\lambda_ht])\Delta \cM(t)| \le &
\,\limsup\limits_{h\to+\infty} |\cL(\lambda_h,[\lambda_ht])\Delta 
\cL(\lambda_h,[\lambda_ht_\epsilon])|\\
&+ \limsup\limits_{h\to+\infty}
|\cL(\lambda_h,[\lambda_ht_\epsilon]) \Delta \cM(t_\epsilon)| + 
 |\cM(t_\epsilon)\Delta \cM(t)| \\
\le &\, 2\widehat c  |t-t_\epsilon|^{\frac{1}{n+1}} 
\le 2\widehat c\, \epsilon
\end{align*}
and the assertion is obtained letting $\epsilon\to0^+.$

Finally, let $\sum\limits_{j=1}^{N+1} 
|\overline{G_j}\setminus G_j| = 0.$ Given $t\in (0,1),$ choosing
$\lambda$ sufficiently large, from \eqref{eq:muhim_holder}  we get
\begin{align*}
|\cL(\lambda,[\lambda t]) \Delta \cL(\lambda,0)| \le &  \,
\cL(\lambda,[\lambda t]) \Delta \cL(\lambda,1)| +
| \cL(\lambda,1)\Delta \cG| \\
\le & \, \widehat c\, 
\Big|t-\frac1\lambda\Big|^{\frac{1}{n+1}} + 
|\cL(\lambda,1)\Delta \cG|.
\end{align*}
Now letting $\lambda\to+\infty$ and
using Proposition \ref{asymptota} a) we establish 
\begin{align*}
|\cM(t)\Delta \cM(0)| \le   \widehat c  \,t^{\frac{1}{n+1}}.
\end{align*}
\end{proof}

In order to improve the H\"older exponent $\frac1{n+1}$ to the value $\frac12$
in \eqref{hulder_est} we expect to be useful, 
for  minimizers $\cA(\lambda)$ of $\func(\cdot,\cA;\lambda),$ 
an  estimate of the form
$$
\sum\limits_{i=1}^{N+1} \sup\limits_{x\in A_i(\lambda)\Delta A_i}
\dist(x,\p A_i) \le O(\lambda^{-1/2}).
$$ 
We miss the  proof of such an estimate; however, 
a partial result in this direction is given in Lemma \ref{lem:L_infty_estimate}.

%%%%%%%%%%%%%%%%%%%%%%%%%%%%%%%%%%%%%%%%%%%%%%%%%%%%%%%%%%%%%%
\section{Existence of $GMM$ in the 
presence of external forces}\label{sec:prescribed_curvature}
%%%%%%%%%%%%%%%%%%%%%%%%%%%%%%%%%%%%%%%%%%%%%%%%%%%%%%%%%%%%%%

In this section we consider the problem of 
the mean curvature evolution of bounded partitions with 
forcing terms. 
Given $\cA\in \P_b(N+1)$ and measurable functions 
$H_i:\R^n\to \R,$ $i=1,\ldots,N+1,$ 
consider the functional 
\begin{equation*}
\funcforce(\cB,\cA;\lambda) = \func(\cB,\cA;\lambda) + 
\sum\limits_{i=1}^{N+1} \int_{B_i} H_idx,\qquad \cB\in\P_b(N+1). 
\end{equation*}

When $N=1$  and $H_2=0,$ we get the 
Almgren-Taylor-Wang functional with an external 
force $H_1.$ 
We suppose: 
\begin{equation}\label{hyp:main}
\begin{cases}
\text{$H_i\in L_\loc^p(\R^n),$ $i=1,\ldots,N+1,$ for some $p>n$ and 
$H_{N+1}\in L^1(\R^n);$ }\\
\text{there exists $R>0$ such that $H_i\ge H_{N+1}$ 
a.e. in $\R^n\setminus B_R(0)$ for any $i=1,\ldots,N;$}
\end{cases}
\end{equation}
in particular $\funcforce(\cdot,\cA;\lambda)$ is well-defined and 
$L^1(\R^n)$-lower semicontinuous.

In the two-phase case ($N=1$), evolutions with  a
forcing term  $H$ depending on both
position and  time have been studied 
for example in  \cite{LS:2016,LS:95} 
(with $H\in C^\infty(\Omega\times [0,T])$ and 
$\Omega\subset\R^n$  bounded), 
%in \cite{LST:2010}
%(with $H=(u\cdot\nu)\nu,$ $\nu$ -- the normal of the boundary
%and the transport term
%$u \in L_\loc^p([0,+\infty),W^{1,p}(\T^n)^n),$ $\T^n$ --
%the $n$-dimensional torus), 
in \cite{CN:2008} (with discontinuous $H$ and
$\int_0^tH(x,s)ds$  locally 
Lipschitz in $x$ and continuous in $t$); 
see also references therein.  

The aim of this section is to prove the following result, 
generalizing Theorem \ref{teo:existence_GMM}.

\begin{theorem}\label{teo:existence_GMM_2}
Suppose that $H_i:\R^n\to \R,$ $i=1,\ldots, N+1,$ satisfy \eqref{hyp:main} and
let  $\cG\in \P_b(N+1).$ Then $GMM(\funcforce,\cG)$ is non empty. Moreover, 
there exists a constant ${\rm C}={\rm C}(N,n,\cG, \allowbreak  p,H_1,\ldots,H_{N+1})>0$
 such that for any $\cM\in GMM(\funcforce,\cG)$
\begin{equation}\label{hulder_est1} 
|\cM(t)\Delta \cM(t')|\le  {\rm C} |t-t'|^{\frac{1}{n+1}},\qquad
\forall t,t'>0,\,\,|t-t'|<1
\end{equation}
and
\begin{equation} \label{uni_bound_convex}
\bigcup\limits_{j=1}^N M_j(t) \subseteq 
\text{closed convex hull of $ {\hull{\cG}\cup B_R(0)}$} 
\qquad \forall t\ge0.
\end{equation} 
In addition, if $\sum\limits_{j=1}^{N+1} 
|\overline{G_j}\setminus G_j| = 0,$ then \eqref{hulder_est1} holds for any 
$t,t'\ge0$ and $|t-t'|<1.$
\end{theorem}

\begin{proof}

{\it Step 1.} Given $\cA\in \P_b(N+1),$ the problem 
$$\inf\limits_{\cB\in\P_b(N+1)} \funcforce(\cB,\cA;\lambda)$$
has a solution.
Let $D$ stand for the closed convex hull of $ {\hull{\cA}\cup B_R(0)}$ 
and for every $\cB\in\P_b(N+1)$ define the competitor  $\cB'\in \P_b(N+1)$ as
$$
\cB':=\Big(B_1\cap D,\ldots,B_N\cap D,\R^n\setminus\bigcup\limits_{i=1}^N (B_i\cap D)\Big).
$$
Observe that 
\begin{equation}\label{eq:represent_funcforce}
\funcforce(\cB,\cA;\lambda) = \func(\cB,\cA;\lambda) + \sum\limits_{j=1}^N \int_{B_j} (H_j- H_{N+1})dx
+\int_{\R^n} H_{N+1}dx. 
\end{equation}
By Remark \ref{rem:cut_convex} we have 
$
\func(\cB,\cA;\lambda) \ge \func(\cB',\cA;\lambda),
$
with the equality if and only if $\big|\bigcup\limits_{j=1}^N B_j\setminus D\big|=0.$
Since $H_j\ge H_{N+1}$ a.e. in $\R^n\setminus D,$ one has also
$$
\sum\limits_{j=1}^N \int_{B_j} (H_j- H_{N+1})dx\ge 
\sum\limits_{j=1}^{N} \int_{B_j\cap D} (H_j- H_{N+1})dx.
$$
Therefore, \eqref{eq:represent_funcforce} implies
$\funcforce(\cB,\cA;\lambda) \ge \funcforce(\cB',\cA;\lambda)$  with
the strict inequality  when $\big|\bigcup\limits_{j=1}^N B_j\setminus D\big|>0.$
Now proceeding as in the proof of Theorem \ref{teo:existence_of_minimizers}
we can show that there  exists a minimizer  of $\funcforce(\cdot,\cA;\lambda).$
Moreover, every minimizer $\cA(\lambda)$ satisfies 
\begin{equation}\label{uniform_boundedness}
\hull{\cA(\lambda)}\subseteq D. 
\end{equation}

\smallskip 
Now we prove the density estimates for $\cA(\lambda).$ 

{\it Step 2.} 
Let us fix $r_0\in (0,R)$ and take any $\cB\in \P_b(N+1)$  with 
$\cA(\lambda)\Delta \cB\strictlyincluded B_r,$ $r\in (0,r_0).$
We show 
\begin{equation}\label{almost_min}
\Per(\cA(\lambda),B_r) \le \Per(\cB,B_r) +
\Lambda_1 |\cA(\lambda)\Delta \cB|^{1-1/p} +
\Lambda_2 |\cA(\lambda)\Delta \cB|,
\end{equation}
where $p$ is given in \eqref{hyp:main} and
\begin{equation}\label{tttuttt}
\Lambda_1:= N^{1/p}\max\limits_{i\le N} \|H_i-H_{N+1}\|_{L^p(D)},\quad 
\Lambda_2:= \frac{\lambda}{2}\,(\diam D+2r_0). 
\end{equation}
Indeed, from \eqref{uniform_boundedness} one has 
$$
\dist(z,\p A_j) \le \diam D  + 2r, \qquad 
j=1,\ldots,N+1,\,\,z\in \cA(\lambda) \Delta \cB,
$$
hence using \eqref{eq:pro_distance}  
\begin{align*}
\Big|\sigma(\cB,\cA) & - \sigma(\cA(\lambda),\cA)\Big| \le 
 \sum\limits_{j=1}^{N+1} \int_{B_j\Delta A_j(\lambda)} \dist(z,\p A_j)dz
\le (\diam D+2r_0)|\cB\Delta \cA(\lambda)|,
\end{align*}
since $\cB\Delta \cA(\lambda)\strictlyincluded B_{r_0}.$
Moreover, from the H\"older inequality
\begin{align*}
\Big|\int_{A_i(\lambda)} & (H_i-H_{N+1})dx -  \int_{B_i} (H_i-H_{N+1})dx\Big|\le 
\int_{A_i(\lambda)\Delta B_i} |H_i-H_{N+1}|dx\\ 
\le & |A_i(\lambda)\Delta B_i|^{1-1/p} 
\Big(\int_{A_i(\lambda)\Delta B_i} |H_i-H_{N+1}|^pdx\Big)^{1/p}\\
\le & 
\|H_i-H_{N+1}\|_{L^p(D)} |A_i(\lambda)\Delta B_i|^{1-1/p}.
\end{align*}
Then the concavity of 
the function $t\in (0,+\infty)\mapsto t^{1-1/p}$ implies that 
\begin{align*}
\Big|\sum\limits_{i=1}^N \int_{A_i(\lambda)} & (H_i-H_{N+1})dx -  
\int_{B_i} (H_i-H_{N+1})dx\Big|\\
\le &   N^{1/p}\max\limits_{i\le N} \|H_i-H_{N+1}\|_{L^p(D)} 
|\cA(\lambda)\Delta \cB|^{1-1/p}. 
\end{align*}
Now minimality of $\cA(\lambda)$ (Step 1)  
yields  \eqref{almost_min}. 

Thus
we can apply  Remark \ref{rem:strange_almost_min} with $\alpha_1=1-1/p>1-1/n,$
$\alpha_2=1,$ $r_0\in (0,R)$ and 
$$
\tilde r_0 = 
\begin{cases}
\min\{r_0, \frac{n}{4\Lambda_2N}\} & \text{if
$\Lambda_1=0,$}\\ 
\min\{r_0, \omega_n^{-1/n} 
\big(\frac{n\omega_n^{1/n}}{8\Lambda_1N}\big)^{\frac{p}{p-n}},
\frac{n}{8\Lambda_2N} \} & \text{if
$\Lambda_1>0,$}
\end{cases}
$$
to get that for every $i\in \{1,\ldots,N+1\},$   
 \eqref{eq:vol.density_est}-\eqref{eq:per.density_est} hold 
for any $x\in \p A_i(\lambda)$ and  $r\in (0,\tilde r_0).$
In particular, 
$
\sum\limits_{j=1}^{N+1} 
\cH^{n-1}(\p A_j(\lambda)\setminus \p^*A_j(\lambda)) =0. 
$

{\it Step 3.} Given $\cG\in \P_b(N+1)$ let $K$ denote the closed convex hull 
of $\hull{\cG}\cup B_R(0).$ Let $\cL(\lambda,0):=\cG$ and $\cL(\lambda,k)$
be defined as 
$$
 \funcforce(\cL(\lambda,k),\cL(\lambda,k-1);\lambda) = \min\limits_{\cA\in \P_b(N+1)}
 \funcforce(\cA,\cL(\lambda,k-1);\lambda),\quad k\ge1.
$$
Notice that by \eqref{uniform_boundedness},
 $\hull{\cL(\lambda,k)}\subseteq K$  for any $\lambda\ge1$ and $k\ge0.$
Observe that for any $\lambda\ge1$ the map
$$
k\in\N_0\mapsto \map(\lambda,k):=\Per(\cL(\lambda,k)) + 
\sum\limits_{j=1}^{N} \int_{L_j(\lambda,k)}
(H_j-H_{N+1})dx
$$
is nonincreasing.
Indeed, since $\funcforce(\cL(\lambda,k),\cL(\lambda,k-1);\lambda)\le 
\funcforce(\cL(\lambda,k-1),\cL(\lambda,k-1);\lambda),$ recalling \eqref{eq:represent_funcforce}
one has 
\begin{align*}
\lambda\sigma(\cL(\lambda,k),\cL(\lambda,k-1)) \le &
2\big(\map(\lambda,k-1) -\map(\lambda,k)\big). 
\end{align*}
In particular, from $\map(\lambda,k)\le \map(\lambda,0)$ it follows that
\begin{equation}\label{eqrere}
\begin{aligned}
\Per(\cL(\lambda,k)) \le & \Per(\cG) +\sum\limits_{j=1}^N \int_{L_j(\lambda,k)\Delta G_j}
|H_j-H_{N+1}|dx \\
\le & \Per(\cG) + N\max\limits_{j\le N} \|H_j-H_{N+1}\|_{L^1(K)} =:\kappa. 
\end{aligned}
\end{equation}

We claim that for every $t,t'>0,$ $0<t-t'<1,$ 
\begin{align}\label{one_two_thr} 
|\cL(\lambda,[\lambda t])\Delta \cL(\lambda,[\lambda t'])|  
\le {\rm C}\, |t-t'|^{\frac{1}{n+1}}
\end{align}
provided that $\lambda\ge \max\{4/t',4/(t-t')\}$ is sufficiently large so that 
the density estimates \eqref{eq:vol.density_est}-\eqref{eq:per.density_est} hold
for $r\in (0,\delta),$ $\delta = \frac{n}{4\lambda N(\diam K +2r_0)},$   
where 
$$
{\rm C}:= {\rm C}(N,n,\cG)= \left(\frac{5^n\omega_n}{c_{n,N+1}}\, 
\frac{n}{2N(\diam K + 2r_0)}  + \frac{8N(\diam K + 2r_0)}{n}\right)\,\Per(\cG ), 
$$
and $c_{n,N+1}$ is defined in \eqref{maaqweod} (with $N+1$ in place of $N$).

Set
$k_0:=[\lambda t'],$ $m_0:= [\lambda t].$ 
By the choice of $\lambda$ we have $m_0\ge k_0+3\ge4.$ 
Note that
\begin{align*}
 \sum\limits_{k=k_0+ 1}^{m_0}
\Big(\map(\lambda,k-1) &-  \map(\lambda,k)\Big)\le  
\map(\lambda,0) - \map(\lambda,m_0)
\le  
\Per(\cG) -\Per(\cL(\lambda,m_0))\\
&+   \sum\limits_{j=1}^N
\Big(\int_{G_j} (H_j - H_{N+1})dx - 
\int_{L_j(\lambda,m_0)} (H_j - H_{N+1})dx \Big)\\
\le & \Per(\cG) + N\max\limits_{j\le N} \|H_j-H_{N+1}\|_{L^1(K)} = \kappa  
\end{align*}
Since  $\cL(\lambda,k),$ $k\ge1,$ satisfies the density estimates  
\eqref{eq:vol.density_est}-\eqref{eq:per.density_est} 
according to Step 2, 
applying Corollary \ref{cor:ATW} with
$\ell=\delta|t-t'|^{-\frac{1}{n+1}},$
we get
\begin{align*}
|\cL(\lambda, & [\lambda t])\Delta \cL(\lambda,[\lambda t'])| \le  
\sum\limits_{k=k_0+ 1}^{m_0} |\cL(\lambda,k)\Delta \cL(\lambda,k-1)|\\
\le  & \frac{5^n\omega_n}{c_{n,N+1}}\,
\frac{n}{4\lambda N(\diam K +2r_0)}\,|t-t'|^{-\frac{n}{n+1}}\,
\sum\limits_{k=k_0+ 1}^{m_0} 
\Per(\cL(\lambda,k-1))\\
&+
\frac{4N(\diam K +2r_0)}{n}\,\lambda \,|t-t'|^{\frac{1}{n+1}}
\,\sum\limits_{k=k_0+ 1}^{m_0} \sigma(\cL(\lambda,k),\cL(\lambda,k-1))\\
\le & \frac{5^n\omega_n}{c_{n,N+1}}\,
\frac{\kappa n}{4N(\diam K +2r_0)}\,|t-t'|^{-\frac{n}{n+1}}\,\frac{m_0-k_0}{\lambda}\\
&+\frac{4N(\diam K +2r_0)}{n}\,|t-t'|^{\frac{1}{n+1}}
\,\sum\limits_{k=k_0+ 1}^{m_0} \Big(\map(\lambda,k-1) -  \map(\lambda,k)\Big).
\end{align*} 
Then \eqref{one_two_thr} follows using \eqref{super_star_a}.

Now the  proofs of \eqref{hulder_est1} and \eqref{uni_bound_convex} are 
exactly the same as in 
the proof of  Theorem \ref{teo:existence_GMM}.

{\it Step 4.} Finally, let us show that if $\sum\limits_{j=1}^{N+1} 
|\overline{G_j}\setminus G_j|=0,$ then \eqref{hulder_est1} holds for any 
$t,t'\ge0,$ $|t-t'|<1.$ We need just to show that 
$|\cL(\lambda,1)\Delta \cG|\to0$ as $\lambda\to+\infty,$ and then we proceed  
as in the proof of the final assertion of Theorem \ref{teo:existence_GMM}.

Using the minimality of $\cL(\lambda,1)$ we have 
$\funcforce(\cL(\lambda,1),\cG;\lambda)\le \funcforce(\cG,\cG;\lambda),$ i.e.
\begin{equation}\label{ppppp}
\frac{\lambda}{2}\, \sigma(\cL(\lambda),\cG) \le  
\Per(\cG) - \Per(\cL(\lambda,1)) + N\max\limits_{j\le N} 
\|H_j-H_{N+1}\|_{L^1(K)} \le \kappa. 
\end{equation}
Choose an arbitrary diverging sequence $\{\lambda_k\}.$
By \eqref{eqrere} it follows  $\Per(\cL(\lambda_k,1))\le \kappa$ for any $k\ge1$
and since $\bigcup\limits_{j=1}^N L_j(\lambda_k,1)\subseteq K,$ 
by Theorem \ref{prop:compactness} there exists a (not relabelled) subsequence
and $\cA\in \P_b(N+1)$ such that $\cL(\lambda_k,1)\to \cA$ in $L^1(\R^n)$
as $k\to+\infty.$ Then the $L^1(\R^n)$-lower semicontinuity of 
$\sigma$ and \eqref{ppppp} yield
$$
\sigma(\cA,\cG) \le \liminf\limits_{k\to+\infty} \sigma(\cL(\lambda_k,1),\cG)
\le \liminf\limits_{k\to+\infty} \frac{2\kappa}{\lambda_k}=0.
$$
Hence $\sigma(\cA,\cG)=0$ and by the assumption of $\cG$ we have $\cA=\cG.$
Since $\{\lambda_k\}$ is arbitrary, $\cL(\lambda,1)\to\cG$ in $L^1(\R^n)$
as $\lambda\to+\infty.$
\end{proof}

%%%%%%%%%%%%%%%%%%%%%%%%%%%%%%%%%%%%%%%%%%%%%%%%%%%%%%%%%%%%%%%%%%%%%%%%%%%%%%%
\section{Evolution of disjoint partitions} \label{sec:evolution_disjoint_part}
%%%%%%%%%%%%%%%%%%%%%%%%%%%%%%%%%%%%%%%%%%%%%%%%%%%%%%%%%%%%%%%%%%%%%%%%%%%%%%%

In this section we study the evolution of disjoint partitions 
and the compatibility results of GMM starting from the disjoint 
initial partition with other notions of solution.

%%%%%%%%%%%%%%%%%%%%%%%%%%%%%%%%%%%%%%%%%%%%%%%%%%%%%%%%%%%%%%%%%%%%%%%%%%%%%%%%%%%%%%% 
\subsection{Some comparison results for the $2$-phase case ($N=1$)} 
%%%%%%%%%%%%%%%%%%%%%%%%%%%%%%%%%%%%%%%%%%%%%%%%%%%%%%%%%%%%%%%%%%%%%%%%%%%%%%%%%%%%%%%

Let us start with recalling some comparison arguments 
for the Almgren-Taylor-Wang functional
$\atw(\cdot,\cdot;\lambda)$ in
\eqref{eq:standard_ATW}  from 
\cite[Section 6]{BH:2016} and \cite[Section 6]{CMP:2015}.

Define
$$
\mathfrak{M}_b:=\{E\in BV(\R^n,\{0,1\}):\,\, \text{$E$ is bounded}\},
$$
$$ 
\mathfrak{M}_u:=\{E\in BV(\R^n,\{0,1\}):\,\, \text{$E^c$ is bounded}\}.
$$
Notice that $\atw(\cdot,\cdot;\lambda)$ is well-defined 
for both $\fM_b$ and $\fM_u.$
The following result is well-known, and is a particular case of Theorem 
\ref{teo:existence_of_minimizers} (applied with $N=1$).  

\begin{proposition}
Given $G\in\mathfrak{M}_b$ (resp. $G\in \fM_u$) and $\lambda\ge1$ 
the problem
$$
\inf\limits_{E\in \fM_b}  \atw(E,G;\lambda)\qquad (\text{resp.} 
\,\,\inf\limits_{E\in \fM_u}  \atw(E,G;\lambda))
$$
has a solution.
Moreover, any minimizer $G(\lambda)$ satisfies the inclusion  
$$
G(\lambda)\subseteq \hull{G}\qquad 
\text{(resp. $\R^n\setminus G(\lambda) \subseteq \hull{\R^n\setminus G}$).}
$$
\end{proposition}

\begin{proposition}[\bf Maximal and minimal minimizers 
{\cite{BH:2016,CMP:2015}}]
Given $E\in \fM_b$ (resp. $E\in \fM_u$) and 
$\lambda\ge1$ there  exist the
maximal and the minimal minimizer  $E(\lambda)^*$ and 
$E(\lambda)_*$ of $\atw(\cdot,E;\lambda),$ in the sense that
any other minimizer $E(\lambda)$ satisfies 
$$
E(\lambda)_* \subseteq E(\lambda) \subseteq E(\lambda)^*.
$$
\end{proposition}

Given a set $E\subset\R^n$ and $\epsilon>0$ we write 
\begin{equation}\label{oppen_nbhrd}
E_\epsilon^+=\{x\in\R^n:\,\, \distance(x,E) < \epsilon\}. 
\end{equation}
We recall the following comparison principles for the minimizers of 
$\atw$ from \cite[section 6]{CMP:2015}, see also \cite[Section 6]{BH:2016}. 

\begin{theorem}[\bf Comparison principles]\label{teo:comparison}
Let $\epsilon>0,$ $E,F\in \fM_b$ (or $E,F\in \fM_u$ or 
$E\in \fM_b$ and $F\in \fM_u$) be such that
\begin{equation}\label{asal_buri}
E_\epsilon^+ \subseteq F. 
\end{equation}
Then 
\begin{equation}\label{zdser}
(E(\lambda))_\delta^+ \subseteq F(\lambda),\qquad \delta<\epsilon, 
\end{equation}
for every $\lambda\ge1$ and every minimizer $E(\lambda)$ and $F(\lambda)$ of 
$\atw(\cdot,E;\lambda)$ and $\atw(\cdot,F;\lambda),$ respectively. 
Moreover,
\begin{equation}\label{shirin_shakalat}
(E(\lambda)_*)_\epsilon^+ \subseteq F(\lambda)_*,
\qquad
(E(\lambda)^*)_\epsilon^+ \subseteq F(\lambda)^*. 
\end{equation}
\end{theorem}

\begin{corollary}\label{cor:disjoint_comparison}
Suppose that $E,F\in \fM_b$ are such that 
$$
\distance(E,F) >  0. 
$$
Then for any $\lambda\ge1,$ 
every minimizer $E(\lambda)$ (resp. $F(\lambda)$) of 
$\atw(\cdot,E;\lambda)$ (resp. $\atw(\cdot,F;\lambda)$) 
satisfies
\begin{equation}\label{deltaqul}
\distance(E(\lambda),F(\lambda)) \ge \distance(E,F).
\end{equation}
\end{corollary}

\begin{definition}[\bf Minimal and maximal GMM associated with a sequ\-ence]
\label{def:min_max_GMM}
For $E\in \fM_b,$  $\{E_*(t)\}\in GMM(\atw,E)$ 
(resp. $\{E^*(t)\}\in  GMM(\atw,E)$) is called  the {\bf minimal}
(resp. the {\bf maximal}) GMM associated with a sequence $\{\lambda_k\}$ 
if 
$$
E(\lambda_k,[\lambda_kt])_* \to E_*(t) \quad\text{(resp. 
$E(\lambda_k,[\lambda_kt])^* \to E_*(t))$ $\quad$ as $k\to+\infty$ in $L^1(\R^n),$}  
$$
where $E(\lambda,0)_*=E(\lambda,0)^* = E,$ and 
$E(\lambda,l)_*$ (resp. $E(\lambda,l)^*$), 
$\lambda\ge1$ and $l\in\N,$ is the minimal (resp. maximal)
minimizer of $\atw(\cdot,E(\lambda,l-1)_*;\lambda)$
(resp. $\atw(\cdot,E(\lambda,l-1)^*;\lambda)$).
\end{definition}

The minimal and maximal GMM satisfy the following comparison theorem 
\cite[Theorem 7.3]{BH:2016}.

\begin{theorem}[\bf Comparison for minimal and maximal GMM]
\label{teo:comparing_GMM}
Let $E,F\in\fM_b,$ $E\subseteq F$  and let $\{E(t)_*\}$
(resp. $\{E(t)^*\}$) be the minimal (resp. maximal) GMM associated with 
a same sequence $\{\lambda_k\}.$ Then 
\begin{equation}
E(t)_*\subseteq F(t)_*\qquad \text{(resp. $E(t)^*\subseteq F(t)^*$) $\qquad$ for all 
$t\ge0.$} 
\end{equation}
\end{theorem}

%%%%%%%%%%%%%%%%%%%%%%%%%%%%%%%%%%%%%%%%%%%%%%%%%%%%%%%%%%%%%%%%%%%%%%%%%%%%%%%
\subsection{Evolution of disjoint partitions} 
%%%%%%%%%%%%%%%%%%%%%%%%%%%%%%%%%%%%%%%%%%%%%%%%%%%%%%%%%%%%%%%%%%%%%%%%%%%%%%%
 
Now we study the evolution of disjoint partitions.

\begin{definition}[\bf Disjoint partitions]\label{def:disjoint}
A partition $\cA\in \P_b(N+1)$  is called disjoint provided 
\begin{equation*}
\min\limits_{1\le i<j\le N} \distance(A_i,A_j)>0. 
\end{equation*} 
\end{definition}

Notice that if $\cA\in \P_b(N+1)$ is disjoint, then
$
\Per(\cA) = \sum\limits_{j=1}^N P(A_j).
$
Moreover, if $\cA $ and $\cG$ are disjoint and satisfy 
\begin{equation}\label{super_disjoint}
\bigcup\limits_{j=1}^N (A_j\Delta G_j) = 
\Big(\bigcup\limits_{j=1}^N A_j\Big)\Delta 
\Big(\bigcup\limits_{j=1}^N G_j\Big), 
\end{equation}
then $\sigma(\cA,\cG) =2\sum\limits_{j=1}^{N} \int_{A_j\Delta G_j} \dist(x,\p G_j)dx$ 
and  
\begin{equation}\label{summa_ATW}
\func(\cA,\cG;\lambda) = \sum\limits_{j=1}^N \Big(P(A_j) + \lambda 
\int_{A_j\Delta G_j} \dist(x,\p G_j)dx\Big) = 
\sum\limits_{j=1}^N \atw(A_j,G_j;\lambda).
\end{equation}

In the next two lemmas, no disjointness hypothesis is assumed.
The proof of the following lemma is an adaptation of the proof 
of Theorem \ref{teo:density_est}.

\begin{lemma}\label{lem:helping_density_e}
Given $\cG\in \P_b(N+1),$ let $\cG(\lambda)\in \P_b(N+1)$ be a
minimizer of $\func(\cdot,\cG;\lambda).$ Fix $i\in\{1,\ldots, N+1\}.$
If $x\in G_i(\lambda)^c\cap G_i$ and 
 $\dist(x,\p G_i)\ge\rho>0,$  then
\begin{equation}\label{asda}
\frac{1}{2^n} \le \frac{|B_\rho(x)\cap G_i(\lambda)^c|}{|B_\rho(x)|}. 
\end{equation}
\end{lemma}

\begin{proof}
Without loss of generality
we suppose $i=1.$ As usual, write $B_r:=B_r(x)$ and set 
$$
I:=\{j\in \{2,\ldots,N+1\}:\,\,\cH^{n-1}(B_\rho\cap \p^* G_1(\lambda)\cap 
\p^*G_j(\lambda))>0\}.
$$  
Clearly, if $I=\emptyset,$ then  
$B_\rho \subseteq G_1(\lambda)^c$ and \eqref{asda} is satisfied,
hence we can suppose $I\ne\emptyset.$ 
Fix any $r\in (0,\rho)$ such that 
\begin{equation}\label{good_rad}
\sum\limits_{j=1}^{N+1} \cH^{n-1}(\p B_r\cap \p^* G_j(\lambda))=0. 
\end{equation}
For each $j\in I$ define the competitor  $\cC^{(j)}\in \P_b(N+1)$ as
\begin{equation}\label{competitor}
\cC^{(j)}:=(G_1(\lambda)\cup (G_j(\lambda)\cap B_r),G_2(\lambda)\ldots,
G_{j-1}(\lambda),G_j(\lambda)\setminus B_r,
G_{j+1}(\lambda),\ldots, G_{N+1}(\lambda)).
\end{equation}
Fix $s\in (r,\rho).$   Arguing as in 
the proofs of \eqref{eq3191} and \eqref{set_operation12},
\begin{align*}
P(G_1(\lambda)\cup(G_j(\lambda)\cap B_r),B_s)= &
P(G_1(\lambda),B_s)+\cH^{n-1}(G_j(\lambda)\cap \p B_r) + P(G_j(\lambda),B_r)\\
&-2\cH^{n-1}(B_r \cap \p^* G_1(\lambda)\cap \p^* G_j(\lambda)),\\
P(G_j(\lambda)\setminus B_r,B_s) = & P(G_j(\lambda),B_s\setminus \overline{B_r}) +
\cH^{n-1}(G_j(\lambda) \cap \p B_r).
\end{align*}
Therefore from \eqref{good_rad}
\begin{align*}
\lim\limits_{s\to r^+} &
\Big(P(G_1(\lambda)\cup(G_j(\lambda)\cap B_r),B_s)   \\
 & +  P(G_j(\lambda)\setminus B_r,B_s)
 - P(G_1(\lambda),B_s) - P(G_j(\lambda),B_s)\Big)\\
= &  2\cH^{n-1}(G_j(\lambda)\cap \p B_r) - 2\cH^{n-1}(B_r\cap \p^*G_1(\lambda)
\cap \p^*G_j(\lambda)).
\end{align*}
Hence the inequality $\func(\cG(\lambda),\cG;\lambda)
\le \func(\cC^{(j)},\cG;\lambda)$ due to the minimality of 
$\cG(\lambda)$ and \eqref{eq:pro_distance} 
imply
\begin{equation}\label{eq:keraksiz}
\begin{aligned}
\cH^{n-1}(G_j(\lambda)\cap \p B_r) - & \cH^{n-1}(B_r\cap \p^*G_1(\lambda)
\cap \p^*G_j(\lambda))\\
\ge &  \frac{\lambda}{2} 
 \int_{G_j(\lambda)\cap B_r} \big(\sdist(y,\p G_j) - 
\sdist(y,\p G_1)\big)dy.
\end{aligned}
\end{equation}
Since by assumption $B_\rho \subseteq G_1$ (and hence $B_\rho\cap G_j =\emptyset$) we have 
\begin{equation}\label{pptop}
\sdist(y,\p G_j) - \sdist(y,\p G_1) = \dist(y,\p G_j) + \dist(y,\p G_1)\ge0 
\qquad \forall y\in G_j(\lambda)\cap B_r,
\end{equation}
and therefore 
\begin{equation}\label{eq:minimial}
 \cH^{n-1}(B_r\cap \p^*G_1(\lambda)\cap \p^*G_j(\lambda))\le 
 \cH^{n-1}(G_j(\lambda)\cap \p B_r).
\end{equation}
Then summation of \eqref{eq:minimial} over $j\in I$ and the
use of  Remark \ref{rem:boundary_neighbor}
yield
\begin{align*}
P(G_{1}(\lambda)^c,B_r)   \le  & \sum\limits_{j\in I} \cH^{n-1}(G_j(\lambda)\cap \p B_r)\le
\sum\limits_{j=2}^{N+1} \cH^{n-1}(G_j(\lambda)\cap \p B_r)\\
 = & \cH^{n-1}(G_{1}(\lambda)^c\cap \p B_r).
\end{align*}
Now adding $\cH^{n-1}(G_{1}(\lambda)^c\cap \p B_r)$  to both sides 
we get 
$$
P(G_{1}(\lambda)^c\cap B_r) \le 2 \cH^{n-1}(G_{1}(\lambda)^c\cap \p B_r).
$$
From the isoperimetric inequality, for a.e. $r\in(0,\rho)$ we obtain
\begin{equation}\label{slsls}
n\omega_n^{1/n}|G_{1}(\lambda)^c\cap B_r|^{\frac{n-1}{n}} \le 
2 \cH^{n-1}(G_{1}(\lambda)^c\cap \p B_r). 
\end{equation}
Since $x\in G_{1}(\lambda)^c,$ one has $|G_{1}(\lambda)^c\cap B_r|>0$ for any $r>0,$
therefore integrating \eqref{slsls} in $(0,\rho),$ 
we get \eqref{asda}.
\end{proof}

\begin{lemma} \label{lem:L_infty_estimate}
Given $\cG\in \P_b(N+1)$ let $\cG(\lambda)\in \P_b(N+1)$ 
be a minimizer of $\func(\cdot,\cG;\lambda).$ 
Then for any $i\in \{1,\ldots,N+1\},$
\begin{equation*}%\label{eq:L_infty_estimate}
\sup\limits_{x\in G_i(\lambda)^c\cap G_i} \dist(x,\p G_i)\le
 \frac{\sqrt{2^{n+2}n}}{\sqrt\lambda}.
\end{equation*}
\end{lemma}

\begin{proof}
Without loss of generality we suppose $i=1.$
By contradiction, let   $x\in G_1(\lambda)^c\cap G_1$ be such that 
$\dist(x,\p G_1) \ge \rho:= \frac{\sqrt{2^{n+2}n} +\epsilon }{\sqrt\lambda}$
for some $\epsilon>0.$ We may suppose that 
$x\in \p G_1(\lambda)$ and $\epsilon$ are such that 
$$
\cH^{n-1}(\p^* G_1(\lambda)\cap \p B_\rho) = 0,
$$
where $B_\rho:=B_\rho(x).$ Then the set 
$$
J:=\{j\in\{2,\ldots,N+1\}:\,\,|B_{\rho/2}\cap G_j(\lambda)|>0\}
$$ 
is nonempty. By assumption on $x$ and $\rho,$
$B_{\rho/2}(y)\subset G_1$ for every $y\in B_{\rho/2},$ 
and hence
$$
\dist(y,\p G_j) \ge \dist(y,\p G_1) \ge \rho/2\qquad
\forall j\in J.
$$
Therefore, for each $j\in J,$ 
defining the competitor as in \eqref{competitor} with $r=\rho/2,$
from the minimality of $\cG(\lambda),$  \eqref{eq:pro_distance}
and \eqref{eq:keraksiz} 
we get 
\begin{align*}
\cH^{n-1}(G_j(\lambda)\cap &\p B_{\rho/2})  - 
\cH^{n-1}(B_{\rho/2}\cap \p^*G_1(\lambda)\cap \p^*G_j(\lambda))\\
& \ge \frac{\lambda }{2}
 \int_{G_j(\lambda)\cap B_{\rho/2}} \big(\sdist(y,\p G_j) - 
\sdist(y,\p G_1)\big)dy \ge \frac{ \lambda\rho }{2} |G_j(\lambda)\cap B_{\rho/2}|,
\end{align*}
since $\sdist(y,\p G_j) = \dist(y,\p G_j)$ and 
$\sdist(y,\p G_1) = -\dist(y,\p G_1)$ for any $y\in B_{\rho/2}.$ 
Summing these inequalities over $j\in J$ and using 
$
\bigcup\limits_{j=1}^{N+1} (G_j(\lambda)\cap B_{\rho/2})
=\bigcup\limits_{j\in J} (G_j(\lambda)\cap B_{\rho/2}) = 
G_{1}(\lambda)^c\cap B_{\rho/2} 
$
(up to a negligible set), we get
$$
\cH^{n-1}(G_{1}(\lambda)^c\cap \p B_{\rho/2}) \ge 
\sum\limits_{j\in J} \cH^{n-1}(B_{\rho/2}\cap \p^*G_1(\lambda)\cap 
\p^*G_j(\lambda)) + 
\frac {\lambda\rho}{2} |G_{1}(\lambda)^c\cap B_{\rho/2}|.
$$
Now Lemma \ref{lem:helping_density_e} yields
$$
\Big(\frac{1}{2}\Big)^{n +1 } \,\lambda\rho \omega_n 
\Big(\frac{\rho}{2}\Big)^n \le \cH^{n-1}(G_{1}(\lambda)^c\cap \p B_{\rho/2}), 
$$
and clearly, $\cH^{n-1}(G_{1}(\lambda)^c\cap \p B_{\rho/2}) 
\le n\omega_n \Big(\frac{\rho}{2}\Big)^{n-1}.$
Therefore,
$
\rho= \frac{\sqrt{2^{n+2}n}+\epsilon}{\sqrt\lambda}\le 
\frac{\sqrt{2^{n+2}n}}{\sqrt\lambda}, 
$
a contradiction, since $\epsilon>0.$
\end{proof}

The following theorem shows that if  
the components of the initial  partition 
$\cG$ are far from each other, then so are the components of minimizers of 
$\func(\cdot,\cG;\lambda),$ provided $\lambda$ is large enough.

\begin{theorem}[\bf Minimizers of $\func $ for a 
disjoint initial partition]\label{teo:disjoint_initial_part}
Suppose that $\cG\in\P_b(N+1)$ is disjoint and set
\begin{equation}\label{eq:disj_pro}
\min\limits_{1\le i<j\le N} \distance(G_i,G_j) =:\epsilon_0>0. 
\end{equation}
Then for   $\lambda > 2^{n+6}n\epsilon_0^{-2}$ any minimizer $\cG(\lambda)$
of $\func(\cdot,\cG;\lambda)$ satisfies 
\begin{equation}\label{hausdorf_nbhd}
G_j(\lambda)\subseteq (G_j)_{\epsilon_0/4}^+,\qquad j=1,\ldots,N. 
\end{equation}
\end{theorem}

\begin{proof}
We claim that the choice of $\lambda$ implies
\begin{equation}\label{epsilon_border}
G_{N+1}(\lambda)^c \subseteq (G_{N+1}^c)_{\epsilon_0/4}^+. 
\end{equation}
Indeed, obviously  $G_{N+1}(\lambda)^c \cap G_{N+1}^c \subseteq 
(G_{N+1}^c)_{\epsilon_0/4}^+.$ 
Now if $x\in G_{N+1}(\lambda)^c \cap G_{N+1},$ then  
$\dist(x, G_{N+1}^c) = \dist(x,\p G_{N+1})$ and therefore
by Lemma \ref{lem:L_infty_estimate} 
$$
\dist(x, G_{N+1}^c) \le 
\sup\limits_{y\in G_{N+1}(\lambda)^c \cap G_{N+1}} \dist(y, \p G_{N+1})\le 
\frac{\sqrt{2^{n+2}n}}{\sqrt\lambda} < \frac{\epsilon_0}{4}.
$$ 
Hence $x\in (G_{N+1}^c)_{\epsilon_0/4}^+$ and \eqref{epsilon_border} follows.

We prove \eqref{hausdorf_nbhd} arguing by contradiction.
Suppose for example $j=1$ and $G_1(\lambda)$ is not contained in 
$(G_1)_{\epsilon_0/4}^+.$ In view of \eqref{epsilon_border}
and \eqref{eq:disj_pro}
\begin{equation}\label{ojjjjj}
G_1(\lambda)\subseteq \bigcup\limits_{j=1}^N G_j(\lambda) \subseteq 
\Big(\bigcup\limits_{j=1}^N G_j\Big)_{\epsilon_0/4}^+= 
\bigcup\limits_{j=1}^N (G_j)_{\epsilon_0/4}^+. 
\end{equation}
Since  $G_1(\lambda)\setminus (G_1)_{\epsilon_0/4}^+ \ne \emptyset,$ 
\eqref{ojjjjj} implies 
$G_1(\lambda)\cap (G_j)_{\epsilon_0/4}^+\ne \emptyset$ 
for some $j\in \{2,\ldots,N\}.$
By virtue of Remark \ref{rem:interior_reg}  the set
$G_1(\lambda)$ can be supposed to be open so that there exists 
a ball $B_r$ of radius $r>0 $ whose closure is 
contained in $ G_1(\lambda)\cap (G_j)_{\epsilon_0/4}^+.$  
For shortness, let $j=2.$
Thus setting $\cB:=(G_1(\lambda)\setminus B_r,G_2(\lambda)\cup B_r,G_3(\lambda),
\ldots,G_{N+1}(\lambda)),$ and using 
$P(G_1(\lambda)) - P(G_1(\lambda)\setminus B_r) = P(B_r),$
we obtain 
\begin{align*}
2\func(\cG(\lambda),\cG;\lambda) - 2\func(\cB,\cG;\lambda)  = &
P(B_r)+ P(G_2(\lambda)) - P(G_2(\lambda)\cup  B_r)\\
&+ 
 \lambda\int_{B_r} \big(\sdist(x,\p G_1) - \sdist(x,\p G_2)\big)dx.
\end{align*}
Now, 
$$
P(B_r)+ P(G_2(\lambda)) - P(B_r\cup G_2(\lambda))\ge0.
$$
In addition, by the definition of $\epsilon_0,$ 
$\dist(\cdot,G_1) \ge \frac{3\epsilon_0}{4}$ in $B_r$
(thus  $\sdist(\cdot,\p G_1)=\dist(\cdot,\p G_1)$ in $B_r$);
moreover, since $B_r\subseteq (G_2)_{\epsilon_0/4}^+,$
one has 
$$
\sdist(x,\p G_1) - \sdist(x,\p G_2) \ge \frac{\epsilon_0}{4}
\qquad\forall x\in B_r
$$
and therefore  
\begin{align*}
\func(\cG(\lambda),\cG;\lambda) - \func(\cB,\cG;\lambda)  \ge &
\frac{\lambda \epsilon_0}{8} |B_r|>0. 
\end{align*}
This implies that $\cG(\lambda)$ is not a minimizer of $\func(\cdot,\cG;\lambda).$
\end{proof}

\begin{corollary}\label{cor:coming_to_ATW}
Suppose that $\cG\in \P_b(N+1)$ is disjoint and let $\epsilon_0$ 
be as in \eqref{eq:disj_pro}.
Then for $\lambda$ sufficiently large (depending only on 
$\epsilon_0$ and $n$\!),
$\cG(\lambda)$ is a minimizer of $\func(\cdot,\cG;\lambda)$ if and only if
each bounded component $G_j(\lambda),$ $j=1,\ldots,N,$ of $\cG(\lambda)$ 
is a minimizer of $\atw(\cdot,G_j;\lambda).$
Moreover, every minimizer $\cG(\lambda)$ satisfies 
\begin{equation}\label{eq:solution_disj}
\min\limits_{1\le i<j\le N} \distance(G_i(\lambda),G_j(\lambda)) \ge\epsilon_0. 
\end{equation}
\end{corollary}

\begin{proof}
By \cite[Lemma 2.1]{LS:95} 
there exists $c(n)>0$ (depending only  on $n$\!) such that  
for every $\lambda\ge1$ and every minimizer 
$A_j(\lambda),$ $j=1,\ldots,N,$  
of $\atw(\cdot,G_j;\lambda)$ one has
$$
\sup\limits_{x\in A_j(\lambda) \Delta G_j}\dist(x,\p G_j) \le 
\sqrt{ \frac{c(n)}{\lambda}}.
$$
Therefore, taking 
\begin{equation}\label{yyzzyy}
\lambda > \tilde c(n)\epsilon_0^{-2},\quad
\tilde c(n):=\max\{2^{n+6}n,16c(n)\}, 
\end{equation}
we deduce 
$A_j(\lambda) \subseteq (G_j)_{\epsilon_0/4}^+,$ 
$j=1,\ldots,N.$  

Set $\cA(\lambda) = (A_1(\lambda),\ldots, 
A_N(\lambda),\R^n\setminus \bigcup\limits_{j=1}^N A_j(\lambda)).$ 
Let us show that for
$\lambda$ as in \eqref{yyzzyy},
$\cA(\lambda)$ minimizes $\func(\cdot,\cG;\lambda).$ Indeed, take any 
minimizer $\cG(\lambda)$ of $\func(\cdot,\cG;\lambda).$ 
By Theorem \ref{teo:disjoint_initial_part} we have
$G_j(\lambda) \subseteq (G_j)_{\epsilon_0/4}^+,$  therefore
both  $(\cA(\lambda),\cG)$ and $(\cG(\lambda),\cG)$ satisfy \eqref{super_disjoint}.
Hence, \eqref{summa_ATW} and the minimality of $A_j(\lambda)$ yield
\begin{align*}
\func(\cG(\lambda),\cG;\lambda) = & 
\sum\limits_{j=1}^N \Big(P(G_j(\lambda)) + \lambda 
\int_{G_j(\lambda)\Delta G_j} \dist(x,\p G_j)dx\Big)\\
\ge & \sum\limits_{j=1}^N \Big(P(A_j(\lambda)) + \lambda 
\int_{A_j(\lambda)\Delta G_j} \dist(x,\p G_j)dx\Big) = 
\func(\cA(\lambda),\cG;\lambda).
\end{align*}
This implies that $\cA(\lambda)$ is also a minimizer $\func(\cdot,\cG;\lambda).$

Conversely, suppose that $\lambda$ satisfies \eqref{yyzzyy}  and
$\cG(\lambda)$ minimizes $\func(\cdot,\cG;\lambda)$
and let $A_j(\lambda),$ $j=1,\ldots,N,$  be a 
minimizer of $\atw(\cdot,G_j;\lambda).$
By \eqref{hausdorf_nbhd},  $A_j(\lambda) \subseteq (G_j)_{\epsilon_0/4}^+,$ 
$j=1,\ldots,N.$ Set $\cA(\lambda) = (A_1(\lambda),\ldots, 
A_N(\lambda),\R^n\setminus \bigcup\limits_{j=1}^N A_j(\lambda)).$  
Then from the minimality of $A_j(\lambda)$ and  $\cG(\lambda),$ as well as
\eqref{summa_ATW}, we deduce
\begin{align*}
\func(\cG(\lambda),\cG;\lambda) \le  & \func(\cA(\lambda),\cG;\lambda) =
\sum\limits_{j=1}^N \Big(P(A_j(\lambda)) + \lambda 
\int_{A_j(\lambda)\Delta G_j} \dist(x,\p G_j)dx\Big)\\
\le & \sum\limits_{j=1}^N \Big(P(G_j(\lambda)) + \lambda 
\int_{G_j(\lambda)\Delta G_j} \dist(x,\p G_j)dx\Big)= 
\func(\cG(\lambda),\cG;\lambda).
\end{align*}
Thus all inequalities are in fact equalities, which is possible
if and only if 
$$
P(G_j(\lambda)) + \lambda 
\int_{G_j(\lambda)\Delta G_j} \dist(x,\p G_j)=
P(A_j(\lambda)) + \lambda 
\int_{A_j(\lambda)\Delta G_j} \dist(x,\p G_j)dx
$$
for any $j=1,\ldots,N.$
Hence, $G_j(\lambda)$ is a minimizer  of $\atw(\cdot, G_j;\lambda).$

Finally, \eqref{eq:solution_disj} directly follows from 
Corollary \ref{cor:disjoint_comparison}.
\end{proof}

\begin{theorem}[\bf Evolution of disjoint partitions]\label{teo:disjoint_evolution}
Assume that $\cG\in \P_b(N+1)$ is disjoint, 
and $\{\cM\}=\{(M_1,\ldots,M_{N+1})\} \in GMM(\func, \cG).$ 
Then $M_i\in GMM(\atw,\allowbreak G_i)$ for any $i=1,\ldots,N.$ In particular, 
there exists $C(n)>0$ such that 
\begin{equation}\label{half_holder}
|\cM(t)\Delta \cM(t')| \le C(n)\,\Per(\cG)\,|t-t'|^{1/2}\qquad 
\forall t,t'>0,\quad |t-t'|<1. 
\end{equation}
\end{theorem}

\begin{proof}
Let $\epsilon_o>0$ be defined as in \eqref{eq:disj_pro}
and take  $c_o:=c_o(n,\epsilon_o)$  
so that Corollary \ref{cor:coming_to_ATW} holds for $\lambda>c_o.$

Let $\cM\in GMM(\func,\cG)$ and let 
\begin{equation}\label{convergencs}
\lim \limits_{l\to+\infty} |\cM(t)\Delta \cL(\lambda_l,[\lambda_lt])| = 
\sum\limits_{i=1}^{N+1} \lim \limits_{l\to+\infty} 
|M_i(t)\Delta L_i(\lambda_l,[\lambda_lt])|=0,\qquad t\ge0, 
\end{equation}
where $\cL(\lambda,k)$ is defined as $\cL(\lambda,0):=\cG,$  and 
$\cL(\lambda,k),$ $k\ge1,$ is a solution of 
$$
\min\limits_{\cA\in \P_b(N+1)} \func(\cA,\cL(\lambda,k-1);\lambda),
$$
and $\{\lambda_l\}_{l\in\N}$ is a diverging sequence.
By induction on $k\ge1,$ and by Corollary \ref{cor:coming_to_ATW},
one can show that 
\begin{equation}\label{eq:disk_disj_M}
\min\limits_{1\le i<j\le N} \distance(L_i(\lambda,k),L_j(\lambda,k)) \ge\epsilon_0  
\end{equation}
for all $\lambda > c_o$  and $k\ge1.$
Therefore, by virtue of Corollary \ref{cor:coming_to_ATW}, for every $k\ge1$
and $\lambda>c_o,$ each $L_i(\lambda,k),$ $i=1,\ldots,N,$  
minimizes $\atw(\cdot,L_i(\lambda,k-1);\lambda).$ Moreover,
from \eqref{convergencs} we obtain
$$
\lim \limits_{l\to+\infty} 
|M_i(t)\Delta L_i(\lambda_l,[\lambda_lt])|=0,\qquad t\ge0. 
$$
Since $L_i(\lambda,0) = G_i,$
from Definition \ref{def:GMM} we obtain $M_i\in GMM(\atw, G_i).$
\smallskip 

Finally, by \cite{BH:2016,LS:95}, there exists 
$C(n)>0$ such that each $M_i\in GMM(\atw,G_i),$ $i=1,\ldots,N,$
satisfies 
\begin{equation}\label{sssfghh}
|M_i(t)\Delta M_i(t')|\le C(n) \,P(G_i)\,|t-t'|^{1/2}\qquad \forall t,t'>0,\,\,
|t-t'|<1.
\end{equation}
Now \eqref{half_holder} follows summing \eqref{sssfghh}
in $i=1,\ldots,N,$ and using 
$
|\cA\Delta\cB| \le 2\sum\limits_{i=1}^N|A_i\Delta B_i|.
$
\end{proof}

\begin{remark}
Let $M_i\in GMM(\atw,G_i),$ $i=1,\ldots,N,$  and 
$\{\lambda_l\}$ be a diverging sequence such that
\begin{equation}\label{compo_convo}
\lim\limits_{l \to+\infty} |M_i(t)\Delta L_i(\lambda_l,[\lambda_l t])| 
= 0,\qquad t\ge0, 
\end{equation}
where $L_i(\lambda,k)$ is defined as $L_i(\lambda,0):=G_i$ and $L_i(\lambda,k),$ $k\ge1,$
is a solution of 
$$
\min\limits_{A\in BV(\R^n,\{0,1\})} \atw(A,L_i(\lambda,k-1);\lambda).
$$
Applying an induction argument on 
$k$ and Corollary \ref{cor:disjoint_comparison},
we establish \eqref{eq:disk_disj_M} for all $\lambda\ge c_o$ and $k\ge1.$
Therefore, again an induction argument on $k\ge1$ 
and Corollary \ref{cor:coming_to_ATW} imply that the partition 
$\cL(\lambda,k)$ defined for such $\lambda$ and $k$ as 
$$
\cL(\lambda,k):=\Big(L_1(\lambda,k),\ldots, L_N(\lambda,k),\R^n\setminus 
\bigcup\limits_{i=1}^N L_i(\lambda,k) \Big)
$$
minimizes $\func(\cdot,\cL(\lambda,k);\lambda).$
Finally, if we denote by $\cM$ the partition whose bounded 
components are $M_i,$ $i=1,\ldots,N,$ then 
by \eqref{compo_convo},
\begin{align*}
\limsup\limits_{l\to+\infty}  |\cL(\lambda_l,[\lambda_lt])\Delta \cM(t)|\le 
2\sum\limits_{i=1}^N
\lim\limits_{l \to+\infty} |M_i(t)\Delta L_i(\lambda_l,[\lambda_l t])| 
= 0,\qquad t\ge0, 
\end{align*}
and hence  $\cM\in GMM(\func,\cG).$
\end{remark}

Now we are in a position to prove Theorem \ref{teo:consistency_intro}.

\begin{proof}
a) follows combining \cite[Theorem 7.4]{ATW93} 
and Theorem \ref{teo:disjoint_evolution}, 
whereas b) is a consequence of Theorem \ref{teo:disjoint_evolution}  and 
\cite[Theorem 4]{Ch:2005}.
\end{proof}

One can say more about the evolution of convex disjoint partitions.

\begin{definition}[\bf Convex disjoint partitions]\label{def:convex_disjoint}
A disjoint partition $\cA\in \P_b(N+1)$  is called convex if the bounded
components of $\cA$ are convex. 
\end{definition}

We define the Hausdorff distance between  two
partitions $\cA,\cB\in\P_b(N+1)$ as
$$
\HD(\cA,\cB): =\sum\limits_{i=1}^N \HD(A_i,B_i), 
$$
where $\HD(A_i,B_i)$ denotes the Hausdorff 
distance between $A_i$ and $B_i$.

\begin{theorem}[\bf Evolution and stability of convex 
disjoint partitions]\label{teo:convex}
Assume that $\cC\in \P_b(N+1)$ is disjoint and convex. 
Then 
$$GMM(\func, \cC)=\{\cM\}=\{(M_1,\ldots,M_{N+1})\}$$ 
is a singleton. 
In particular, for any $i,j\in\{1,\ldots,N\},$ $i\ne j,$  the function 
\begin{equation}\label{sssdff}
t\in [0,\min\{ t_i^\dagger, t_j^\dagger\})\mapsto 
\distance(M_i(t),M_j(t))  
\end{equation}
is nondecreasing. 
Moreover, for any $i=1,\ldots,N,$ $M_i(\cdot)$ 
agrees with the classical mean curvature flow starting from 
$C_i$ up to its extinction time $t_i^\dag$  \cite{Hui:84}.
Finally, if the sequence $\{\cG^{(h)}\}\subset\P_b(N+1)$ 
converges  to $\cC$ in the Hausdorff distance $\HD$ as $h\to+\infty,$ then 
for any $\cM^{(h)} \in  GMM(\func,\cG^{(h)}),$
\begin{equation*}%\label{hd_convergence}
\lim\limits_{h\to+\infty} \HD(\cM^{(h)}(t),\cM(t)) = 
0 \qquad\forall t\in 
[0,\min\limits_{i\le N}t_i^\dagger).
\end{equation*}
\end{theorem}

\begin{proof}
The first part of the theorem follows 
from Theorem \ref{teo:disjoint_evolution}
and \cite[Corollary 5]{BCCN:05}. 
Before proving the second part of the theorem,
 we show the following stability property of convex sets.
\\[1mm]
{\noindent
{\it Claim.} Let $C\subset\R^n$ be a nonempty 
bounded convex set and let 
$\{G^{(h)}\}$ be a sequence of sets of finite perimeter 
converging to $C$ in the Hausdorff distance as 
$h\to+\infty.$ Then 
\begin{equation}\label{eeeerrrr}
G^{(h)}(t)\overset {\HD}{\to} C(t),\quad t\in [0,t_C^\dag),
\end{equation}
where $G^{(h)}(t)$ and $C(t)$ are Almgren-Taylor-Wang solutions 
starting from $G^{(h)}$ and $C$ respectively (recall that $C(\cdot)$ is 
unique \cite[Corollary 5]{BCCN:05}),  and $t_C^\dag$  is 
the extinction time of $C.$ }

Indeed, consider arbitrary 
sequences $\{A^{(l)}\},$ $\{B^{(l)}\}$ of nonempty bounded 
convex sets such that 
$A^{(l)}\strictlyincluded C \strictlyincluded B^{(l)},$ $l\ge1,$ 
and $A^{(l)},B^{(l)}\overset{\HD}{\to} C$ as $l\to+\infty.$
Then for any $l\ge1,$  there exists $h_l\in\N$ such that 
$A^{(l)}\subseteq G^{(h)}\subseteq B^{(l)}$ for any $h>h_l.$
We may suppose that $h_l\to+\infty$ as $l\to+\infty.$
Let $A^{(l)}(t)$  (resp. $B^{(l)}(t)$)  be the minimizing movement 
starting from $A^{(l)}$  (resp. $B^{(l)}$)  for the 
Almgren-Taylor-Wang functional \eqref{eq:standard_ATW}
and $G^{(h)}(t)^*$ and $G^{(h)}(t)_*$ 
be the maximal and minimal $GMM$\!s (Definition \ref{def:min_max_GMM})
for \eqref{eq:standard_ATW} starting from $G^{(h)},$  
so that ${G^{(h)}}(t)_*\subseteq G^{(h)}(t)\subseteq {G^{(h)}}(t)^*$
for all $t\ge0.$  By  Theorem \ref{teo:comparing_GMM},
$A^{(l)}(t) \subseteq {G^{(h)}}(t)_*$ and 
${G^{(h)}}(t)^* \subseteq B^{(l)}(t)$ for any $t\ge0$ and $h> h_l.$
Moreover, from \cite[Theorem 12]{BCCN:05}  
we have $A^{(l)}(t),B^{(l)}(t)\overset{\HD}{\to} C(t)$
as $l\to+\infty$ for any $t\in [0,t_C),$ and since $h_l\to+\infty,$
\eqref{eeeerrrr} follows.

Now we prove the second part of Theorem \ref{teo:convex}. 
Since the partition $\cC$ is disjoint and $\HD(\cG^{(h)},\cC)\to0$ 
as $h\to+\infty,$ one has that $\cG^{(h)}$ is also disjoint provided 
$h$ is large enough.  Let $\cM^{(h)} \in GMM(\func,\cG^{(h)});$
by Theorem \ref{teo:disjoint_evolution} 
$M_i^{(h)} \in GMM(\atw,G_i^{(h)}),$ $i=1,\ldots,N,$ and 
therefore by virtue of $G_i^{(h)}\overset{\HD}{\to} C_i$ and 
the previous claim, 
$M_i^{(h)}(t)\overset{\HD}{\to} M_i(t),$ $i=1,\ldots,N,$  as $h\to+\infty$
for any $t\in[0,t_i^\dag).$
%Let $\cA\in\P_b(N+1)$
%be a convex disjoint partition with $C_i\strictlyincluded A_i,$
%$i=1,\ldots,N.$  Then for sufficiently large $h,$ 
%$G_i^{(h)}\subset A_i.$ Let 
%$\cG^{(h)}(\lambda_{h,k},[\lambda_{h,k}t])$
%be the sequence chosen in the definition of 
%$\cM^{(h)}(\cdot),$ i.e. $\cG^{(h)}(\lambda_{h,k},[\lambda_{h,k}t])$
%minimizes $\func(\cdot,\cG^{(h)}(\lambda_{h,k},[\lambda_{h,k}t]-1);\lambda_{h,k})$
%and $\cG^{(h)}(\lambda_{h,k},[\lambda_{h,k}t])\to \cM^{(h)}(t)$ 
%in $L^1(\R^n)$ as $k\to+\infty$ for any $t\ge0.$ 
%By  Corollary \ref{cor:coming_to_ATW}, 
%each $G_i^{(h)}(\lambda_{h,k},[\lambda_{h,k}t]),$ $i=1,\ldots,N,$ 
%minimizes \eqref{eq:standard_ATW} (provided $k$ is large enough) with 
%$G_i^{(h)}(\lambda_{h,k},[\lambda_{h,k}t]-1)$ in place of $G,$  
%therefore, $M_i^{(h)}(\cdot)$ is an Almgren-Taylor-Wang
%solution starting from $G_i^{(h)}.$ Now since $G_i^{(h)}\overset{\HD}{\to} C_i,$
%the previous claim implies 
%
%
\end{proof}

\subsection*{Acknowledgements}
%The first author would like to express his gratitude to the
%International Centre for Theoretical Physics (ICTP) in Trieste
%for its hospitality and facilities. 
The first author is %also 
partially supported by GNAMPA of INdAM.
% The second author is very grateful to the
%International Centre for Theoretical Physics (ICTP) and the
%Scuola Internazionale Superiore di Studi Avanzati (SISSA) in Trieste,
%where this research was made.


\begin{thebibliography}{sh}

\bibitem{ATW93}  {\sc F. Almgren, J. Taylor, L. Wang:} Curvature-driven
 flows: a variational approach. SIAM J. Control Optim.  {\bf 31} (1993), 387-438.

 
\bibitem{LAln} {\sc L. Ambrosio:} Movimenti minimizzanti. Rend. Accad. Naz. Sci.
XL Mem. Mat. Appl.  {\bf 113} (1995), 191-246.


\bibitem{AFP:00}  {\sc L. Ambrosio, N. Fusco, D. Pallara:} Functions of Bounded Variation
and Free Discontinuity Problems. Oxford University Press, New York, 2000.
 
\bibitem{AGS:05} {\sc L. Ambrosio, N. Gigli, G. Savar\'e:} Gradient Flows
in Metric Spaces and in the Space of Probability Measures. 
Birkh\"auser-Verlag, Basel, 2008.
 
\bibitem{BKPS:1999} {\sc J. Ball, D. Kinderlehrer, P. Podio-Guidugli, M. Slemrod:} 
Fundamental Contributions to the Continuum Theory of Evolving Phase 
Interfaces in Solids. Springer-Verlag, Berlin, 1999.

\bibitem{Bell:2012} {\sc G. Bellettini:}
Lecture Notes on Mean Curvature Flow, Barriers and Singular Perturbations.
Publications of the Scuola Normale Superiore di Pisa, Vol. 12, 2013.


\bibitem{BCCN:05} {\sc G. Bellettini, V. Caselles, A. Chambolle,
M. Novaga:} Crystalline mean curvature flow of convex sets.
Arch. Ration. Mech.  Anal. {\bf 179} (2006), 109-152.


\bibitem{BH:2016} {\sc G. Bellettini, Sh. Kholmatov:} 
Minimizing movements for mean curvature flow of droplets with prescribed 
contact angle. J. Math. Pures Appl., to appear.

\bibitem{Br:1978} {\sc K. Brakke:} The Motion of a Surface by its Mean Curvature.
Math. Notes, Vol. 20. Princeton University Press, Princeton, 1978.


\bibitem{Car:Th} {\sc D. Caraballo:}
A variational scheme for the evolution of polycrystals by curvature. Ph.D. thesis, 
Princeton University, 1996.

\bibitem{Ch:2005} {\sc A. Chambolle:} An algorithm for mean curvature motion.
Interfaces Free Boundaries {\bf 6} (2004), 195-218. 


\bibitem{CMP:2015} {\sc A. Chambolle, M. Morini, M. Ponsiglione:}
Nonlocal curvature flows. Arch. Ration. Mech. Anal. {\bf 218} (2015), 1263-1329.

\bibitem{CN:2008} {\sc A. Chambolle, M. Novaga:} 
Implicit time discretization of the mean curvature flow with
a discontinuous forcing term. Interfaces  Free Boundaries {\bf 10} 
(2008), 283-300.



\bibitem{CM:book} {\sc T. Colding, W. Minicozzi II:} A Course in Minimal Surfaces. 
Graduate Studies in Mathematics, {\bf12}, AMS, RI, 2011.

\bibitem{DG:58-1} {\sc E. De Giorgi:} Sulla propriet\`a isoperimetrica dell'ipersfera,
nella classe degli
insiemi aventi frontiera orientata di misura finita. Atti Accad. Naz. Lincei Mem.
Cl. Sci. Fis. Mat. Nat. Sez. I (8), {\bf 5} (1958), 33-44.

\bibitem{DG:61} {\sc E. De Giorgi:} Complementi alla teoria della misura $(n-1)$-dimensionale
in uno spazio $n$ dimensionale. Sem. Mat. Scuola Norm. Sup. Pisa,
1960-61. Editrice Tecnico Scientifica, Pisa, 1961.



\bibitem{DG:93} {\sc E. De Giorgi:} New problems on minimizing movements. Boundary value problems for
partial differential equations and applications.
RMA Res. Notes Appl. Math.  {\bf29} (1993), 81-98, Masson, Paris.

\bibitem{DG-96} {\sc E. De Giorgi:} Movimenti di partizioni. Progress in Nonlinear 
Differential Equations and their Applications  {\bf 25} (1996), 1-4.


\bibitem{DGK:14} {\sc D. Depner, H. Garcke, Y. Kohsaka:}
Mean curvature flow with triple junctions in higher space dimensions. Arch. Ration. Mech. Anal. 
{\bf 211} (2014), 301-334.


\bibitem{Eck:2004} {\sc K. Ecker:} Regularity Theory for Mean Curvature 
Flow. Birkh\"auser, Basel, 2004.

\bibitem{EO:2014} {\sc S. Esedo\={g}lu, F. Otto:} Threshold dynamics 
for networks with arbitrary surface tensions. Comm. Pure Appl. Math.
{\bf68} (2015), 808-864.

\bibitem{ESS:1992} {\sc L. Evans, H. Soner, P. Souganidis:} Phase transitions and
generalized motion by mean curvature. Comm. Pure Appl. Math.  {\bf 45} (1992), 1097-1123.

\bibitem{FFL:2006} {\sc F. Ferrari, B. Franchi, G. Lu:} 
On a relative Alexandrov-Fenchel inequality for convex bodies 
in Euclidean spaces. Forum Math. {\bf 18} (2006), 907-921.



\bibitem{Fr:2010} {\sc A. Freire:} Mean curvature motion of graphs with
constant contact angle at a free
boundary. Anal. PDE   {\bf 3} (2010), 359-407.


\bibitem{Fr:20-2} {\sc A. Freire:} Mean curvature motion of triple 
junctions of graphs in two dimensions.
Comm. Partial Differential Equations  {\bf 35} (2010),  302-327.

\bibitem{GH:86} {\sc M. Gage, R. Hamilton:} The heat equation 
shrinking convex plane curves. J. Differ.
Geom. {\bf 23} (1986), 69-95. 


\bibitem{Giga:06} {\sc Y. Giga:} Surface Evolution Equations. Birkh\"auser, Basel, 2006.
 


\bibitem{Gius84} {\sc E. Giusti:} Minimal Surfaces and Functions of Bounded Variation.
Birkh\"auser, Basel, 1984.



\bibitem{Hui:84} {\sc G. Huisken:} Flow by mean curvature of convex surfaces into spheres. 
J. Differ. Geom.
{\bf 20} (1984), 237-266.


\bibitem{Il:94} {\sc T. Ilmanen:} Elliptic Regularization and Partial
Regularity for Motion by Mean Curvature.
Mem. Amer. Math. Soc. {\bf108}, AMS, 1994.



\bibitem{KT:2016} {\sc L. Kim, Y. Tonegawa:}
On the mean curvature flow of grain boundaries. 
Ann. Inst. Fourier (Grenoble) {\bf 67} (2017), 43-142.

\bibitem{KL:2001} {\sc D. Kinderlehrer, C. Liu:} Evolution of grain
 boundaries. Math. Models Methods Appl.
Sci.  {\bf11} (2001), 713-729.

\bibitem{Kur:1966} {\sc K. Kuratowski:} Topology. Vol. 1, Academic Press, New 
York and London, 1966.


\bibitem{LO:2016} {\sc T. Laux, F. Otto:} Convergence of the thresholding scheme for
multi-phase mean-curvature flow. Calc. Var. Partial Differential Equations 
{\bf 55} (2016), 55-129.

\bibitem{LS:2016} {\sc T. Laux, D. Swartz:} Convergence of thresholding schemes
incorporating bulk effects. Interfaces Free Boundaries {\bf 19} (2017), 273-304.


\bibitem{LT:2002} {\sc G. Leonardi, I. Tamanini:} Metric spaces of 
partitions, and Caccioppoli partitions. Adv. Math. Sci. Appl.  {\bf 12} 
(2002), 725-753.

\bibitem{LST:2010} {\sc Ch. Liu, N. Sato, Y. Tonegawa:} 
On the existence of mean curvature flow with transport term.
Interfaces  Free Boundaries {\bf 12} (2010), 251-277.


\bibitem{LS:95} {\sc S. Luckhaus, T. Sturzenhecker:} Implicit time discretization for
the mean curvature flow equation.
Calc. Var. Partial Differential Equations {\bf 3} (1995), 253-271.


\bibitem{Mag12} {\sc F. Maggi:} Sets of Finite Perimeter and Geometric Variational Problems.
An Introduction to Geometric Measure Theory. Cambridge University Press, Cambridge, 2012.

\bibitem{Man:2011} {\sc C. Mantegazza:} Lecture Notes on 
Mean Curvature Flow. Birkh\"auser, Basel, 2011.

\bibitem{MNPS:2016} {\sc C. Mantegazza, M. Novaga, A. Pluda, F. Schulze:}
Evolution of networks with multiple junctions. arXiv:1611.08254 [math.DG].

\bibitem{MBO:1992} {\sc B. Merriman, J. Bence, S. Osher:} 
Diffusion generated motion by mean curvature. Department of
Mathematics, University of California, Los Angeles, 1992.

\bibitem{MBO:1994} {\sc B. Merriman, J. Bence, S. Osher:} 
 Motion of multiple junctions: a level set approach. 
 J. Comput. Phys. {\bf 112} (1994), 334-363. 

\end{thebibliography}
\end{document}